\documentclass{article}
\usepackage{amsthm,amsmath,amssymb,amsfonts}
  \newtheorem{thm}{Theorem}[section]
  \newtheorem{lem}[thm]{Lemma}
  \newtheorem{prop}[thm]{Proposition}
  \newtheorem{cor}[thm]{Corollary}
  \theoremstyle{definition}
  \newtheorem{defn}[thm]{Definition}
  \newtheorem{exm}[thm]{Example}
  \newtheorem{rmk}[thm]{Remark}
  
 \newcommand\ra{\rightarrow}

 \newcommand\s{\subseteq}
 \numberwithin{equation}{section}
\begin{document}
\title {$L$-Ordered and $L$-Lattice Ordered Groups}
\author{R. A. Borzooei$^{1}$, A. Dvure\v{c}enskij$^{2,3}$, and O. Zahiri$^{1}$ \\
{\small\em  $^{1}$ Department of Mathematics, Shahid Beheshti University, G. C., Tehran, Iran}\\
{\small\em $^2$ Mathematical Institute,  Slovak Academy of Sciences, \v Stef\'anikova 49, SK-814 73 Bratislava, Slovakia} \\
{\small\em $^3$ Depart. Algebra  Geom.,  Palack\'{y} Univer.17. listopadu 12,
CZ-771 46 Olomouc, Czech Republic} \\
{\small\em  borzooei@sbu.ac.ir \quad dvurecen@mat.savba.sk\quad   om.zahiri@gmail.com} }
\date{}
\maketitle


\begin{abstract}
This paper pursues an investigation on groups equipped with an $L$-ordered relation, where $L$ is a fixed complete complete Heyting algebra.
First, by the concept of join and meet on an $L$-ordered set, the notion of an $L$-lattice is introduced and  some related results are obtained. Then we applied them to define an $L$-lattice ordered group. We also introduce convex $L$-subgroups to construct a quotient $L$-ordered group.
At last, a relation between the positive cone of an $L$-ordered group
and special type of elements of $L^G$ is found, where $G$ is a group.
\end{abstract}

{\small Mathematics Subject Classification: 06A15, 03E72, 06D72, 06D20.}

Keywords: Fuzzy order; $L$-frame; $L$-lattice ordered group.


\section{Introduction}

Fuzzy orders can be divided into two groups. The first group follows  Zadeh's paper \cite{Zadeh}
and uses an antisymmetry condition based on a crisp relation of equality (see e.g. \cite{F1,F2,F3,F4}),
the second one (see e.g. \cite{G1,Bel1,Bel2,G2,G3})
involves a fuzzy relation of equivalence or its specific case of fuzzy equality.
The former is more frequent, the latter is younger (it was initiated by H\"{o}hle in \cite{Hol}).
Mutual transformations of some representatives of both groups were described e.g. in \cite{XZF,YL}.

In \cite{ZL}, Zhang and Liu defined a kind of an $L$-frame by a pair $(A,i_A)$, where A is a classical frame and
$i_A:L\ra A$ is a frame morphism. For a stratified $L$-topological space $(X,\delta)$, the pair $(\delta,i_X)$
is one of this kind of $L$-frames, where $i_X:L\ra \delta$, is a map which sends $a\in L$ to the constant map with the value $a$.
Conversely, a point of an $L$-frame $(A,i_A)$ is a frame morphism $p:(A,i_A)\ra (L,id_L)$ satisfying $p\circ i_A=id_L$ and
$Lpt(A)$ denotes the set of all points of $(A,i_A)$. Then $\{\Phi_x:Lpt(A)\ra L|\ \forall\, p\in Lpt(A),\ \Phi_x(p)=p(x)\}$
is a stratified $L$-topology on $Lpt(A)$. By these two assignments, Zhang and Liu constructed an adjunction between
$SL-Top$ and $L-Loc$ and consequently they established the Stone representation theorem for distributive lattices by means of this adjunction.
They pointed out that, from the viewpoint of lattice theory, Rodabaugh's fuzzy version of the Stone
representation theory is just one and it has nothing different from the classical one. While in our opinion, Zhang-Liu's $L$-frames
preserve many features and also seem to have no strong difference from a crisp one.

Recently, based on complete Heyting algebras, Fan and Zhang \cite{Fan,ZF},
studied quantitative domains through fuzzy
set theory. Their approach first defines a fuzzy partial order, specifically a degree function, on a non-empty set.
Then they define and study fuzzy directed subsets and (continuous) fuzzy directed complete posets (dcpos for short).
Moreover, Yao \cite{YSI} and   Yao and Shi \cite{Y2S} pursued an investigation on quantitative domains via fuzzy sets. They defined the notions of
fuzzy Scott topology on fuzzy dcpos, Scott convergence and topological convergence for stratified $L$-filters and study them. They also,
shown that the category of fuzzy dcpos with fuzzy Scott continuous maps are Cartesian-closed.

In \cite{Y2}, Yao introduced an $L$-frame by an $L$-ordered set equipped with some further conditions. It is a
complete $L$-ordered set with the meet operation having a right fuzzy adjoint. They established an
adjunction between the category of stratified $L$-topological spaces and the category of $L$-locales,
the opposite category of this kind of $L$-frames.  Moreover, Yao and Shi, \cite{Yao-Shi},
defined on fuzzy dcpos an $L$-topology, called the fuzzy Scott topology, and then they studied its
properties. They defined Scott convergence, topological convergence for stratified $L$-filters and showed that
a fuzzy dcpo is continuous if and only if, for any stratified $L$-filter, the fuzzy Scott convergence coincides with convergence
with respect to the fuzzy Scott topology.

The content of this paper is organized as follows. In Section 2, some notions and results about ordered groups and $L$-ordered sets
are recalled. In Section 3, the concept of an $L$-lattice introduced and some related results are obtained. In Section 4,
the notions of an $L$-ordered group, a convex $L$-subgroup and positive cone of an $L$-ordered group are introduced. Then
a relation between positive cones and $L$-ordered relations of a group is verified and a quotient
$L$-ordered groups is constructed by a convex normal $L$-subgroups. Also, we state a general form for Riesz's decomposition
property in $L$-lattice ordered groups.
\section{Preliminaries}

Let $(L;\vee,\wedge,0,1)$  be a bounded lattice. For $a,b\in L$, we say that $c\in L$ is a {\it relative pseudocomplement} of $a$
with respect to $b$ if $c$ is the largest element with $a\wedge c\leq b$ and we denote it by $a\ra b$.
A lattice $(L;\vee,\wedge)$ is said to be a {\it Heyting} algebra if the relative pseudocomplement $a\to b$ exists for all elements $a,b \in L$.
A {\it frame} is a complete lattice $(L;\vee,\wedge)$ satisfying the infinite distributive law $a\wedge \bigvee S=\bigvee_{s\in S}(a\wedge s)$ for
every $a\in L$ and $S\subseteq L$. It is well known that $L$ is a frame if and only if it is a complete Heyting algebra.
In fact, if $(L;\vee,\wedge)$ is a frame, then for each $a,b\in L$, the relative pseudocomplement of $a$ with respect to $b$,  is the element
$a\ra b:=\vee\{x\in L|\ a\wedge x\leq b\}$. In the following we list some important properties of Heyting algebras,
for more details relevant to frames and Heyting algebras we refer to \cite{Johnston} and \cite[Section 7]{Blyth}:

\begin{enumerate}
\item[(i)]\ $(x\wedge y)\ra z=x\ra (y\ra z)$;
\item[(ii)]\ $x\ra (y\wedge z)=(x\ra y)\wedge (x\ra z)$;
\item[(iii)]\ $(x\vee y)\ra z=(x\ra z)\wedge (y\ra z)$.
\end{enumerate}

From now on, in this paper, $(L;\vee,\wedge,0,1)$ or simply $L$ always denotes a frame.
\begin{defn}\cite{Bel1,Bel2,ZF,ZF2}
Let $P$ be a set and $e : P\times P\ra L$ be a map. The pair $(P;e)$ is called an {\it $L$-ordered set} if for all $x,y,z\in P$
\begin{itemize}
 \item[{\rm(E1)}] $e(x,x)=1$;
 \item[{\rm(E2)}] $e(x,y)\wedge e(y,z)\leq e(x,z)$;
 \item[{\rm(E3)}] $e(x,y)=e(y,x)=1$ implies $x=y$.
\end{itemize}
In an $L$-ordered set $(P;e)$, the map $e$ is called an {\it $L$-order relation} on $P$.
If $(P;\leq)$ is a classical poset, then $(P;\chi_{\leq})$ is an $L$-ordered set, where $\chi_{\leq}$ is the
characteristic function of $\leq$. Moreover, for each $L$-ordered set $(P;e)$,
the set $\leq_e=\{(x,y)\in P\times P|\ e(x,y)=1\}$
is a crisp partial order on $P$ and $(P;\leq_e)$ is a poset. Assume that $(P;e)$ is an $L$-ordered set and $\phi\in L^P$.
Define $\downarrow \phi\in L^P$ and $\uparrow \phi\in L^P$ \cite{Fuentes,ZF}, as follows:
 $$\downarrow\phi(x)=\vee_{x'\in P}(\phi(x')\wedge e(x,x')), \quad \quad \uparrow\phi(x)=\vee_{x'\in P}
 (\phi(x')\wedge e(x',x)),\quad \forall x\in P.$$
\end{defn}
\begin{defn}\cite{Y1,YL}
 A map $f:(P;e_P)\ra (Q;e_Q)$ between two $L$-ordered sets is called {\it monotone} if for all $x,y\in P$,
$e_P(x,y)\leq e_Q(f(x),f(y))$.
\end{defn}

\begin{defn}\cite{Y1,YL,YSI}
Let $(P;e_P)$ and $(Q;e_Q)$ be two $L$-ordered sets and $f:P\ra Q$ and $g:Q\ra P$ be two monotone maps. The pair
$(f,g)$ is called a {\it fuzzy Galois connection} between $P$ and $Q$ if $e_Q(f(x),y)=e_P(x,g(y))$
for all $x\in P$ and $y\in Q$, where $f$ is called the {\it fuzzy left adjoint} of $g$ and dually $g$ the {\it fuzzy right adjoint} of $f$.
\end{defn}

\begin{defn}\cite{Y1,YL}
Let $(P;e)$ be an $L$-ordered set and $S\in L^P$. An element $x_0$ is called a {\it join} (respectively, a {\it meet}) of $S$, in symbols $x_0=\sqcup S$
(respectively, $x_0=\sqcap S$),  if for all $x\in P$,
\begin{itemize}
\item[{\rm (J1)}] $S(x)\leq e(x,x_0)$ (respectively, (M1), $S(x)\leq e(x_0,x))$;
\item[{\rm (J2)}] $\wedge_{y\in P} (S(y)\ra e(y,x))\leq e(x_0,x)$ (respectively, (M2), $\wedge_{y\in P} (S(y)\ra e(x,y))\leq e(x,x_0))$.
\end{itemize}
If a join and a meet of $S$ exist, then they are unique (see \cite{ZF}).
\end{defn}

\begin{thm}\label{join and meet}{\em\cite[Theorem 2.2]{ZF}}
Let $(P;e)$ be an $L$-ordered set and $S\in L^P$. Then
\begin{itemize}
\item[{\rm (i)}] $x_0=\sqcup S$ if and only if $e(x_0,x)=\wedge_{y\in P} (S(y)\ra e(y,x))$;
\item[{\rm (ii)}] $x_0=\sqcap S$ if and only if $e(x,x_0)=\wedge_{y\in P} (S(y)\ra  e(x,y))$.
\end{itemize}
\end{thm}

In \cite{ZXF}, Zhang  et al. introduced a complete $L$-ordered set.
An $L$-ordered set $(P;e)$ is called {\it complete} if, for all $S\in L^P$, $\sqcap S$
and $\sqcup S$ exist. If $(P;e)$ is a complete
$L$-ordered set, then $(P;\leq_e)$ is a complete lattice,
where $\vee S=\sqcup \chi_S$ and $\wedge S=\sqcap\chi_S$, for any $S\s P$.

Suppose that $X$ and $Y$ are two sets. For each mapping $f:X\ra Y$,
we have a mapping $f^{\ra}:L^X\ra L^Y$, defined by
$$
(\forall y\in Y)(\forall A\in L^X) (f^{\ra}(A)(y)=\vee_{f(x)=y}A(x)).
$$
For simplicity, we use $f(A)$ instead of $f^{\ra}(A)$ for all $A\in L^X$.

\begin{thm}\label{adjoint}{\em\cite[Theorem 3.5]{YL}}
Let $f:(P,e_P)\ra (Q,e_Q)$ and $g:(Q,e_Q)\ra (P,e_P)$ be two maps.
\begin{itemize}
 \item[{\rm (i)}] If $P$ is complete, then $f$ is monotone and has a fuzzy right adjoint if and only if
 $f(\sqcup S)=\sqcup f^{\ra}(S)$ for all $S\in L^P$.
 \item[{\rm (ii)}] If $Q$ is complete, then $g$ is monotone and has a fuzzy left adjoint if and only if
 $g(\sqcap S)=\sqcap g^{\ra}(S)$ for all $S\in L^Q$.
\end{itemize}
\end{thm}

\begin{defn}\cite{Y2}
Let $(P;e)$ be a complete $L$-ordered set and $\wedge$ be the meet operation on $(P;\leq_e)$ such that, for any $a\in P$,
the map $\wedge_a()$, $b\mapsto a\wedge b$, is monotone. We call $(P;e)$ an $L$-{\it frame} if, for any $a\in P$,
$\wedge_a$ has a fuzzy right adjoint, or equivalently, the following identity holds:

${\rm(FIDL)}$ \ \ $\wedge_a \sqcup S=\sqcup(\wedge_a)^{\ra}(S)$ for all $S\in L^P$ and $a\in P$.
\end{defn}

\begin{defn}\cite{Blyth}
An {\it ordered group} is an structure $(G;+,0,\leq)$ such that $(G;+,0)$ is a group and $(G;\leq)$ is a poset satisfying
the following condition:
\begin{itemize}
\item for any $x,y,a\in G$, $x\leq y$ implies that $a+ x\leq a+ y$ and $x+ a\leq y+ a$.
\end{itemize}
By a {\it lattice ordered group} we shall mean an ordered group $(G;+,0,\leq)$ such that $(G;\leq)$ is a lattice.
\end{defn}

\begin{thm}{\em\cite{Blyth}}\label{group-dis}
If $(G;+,0,\leq)$ is a lattice ordered group, then for every $x\in G$, the
translation maps $x\wedge -:y \longmapsto x\wedge y$, $x\vee -:y\longmapsto x\vee y$,  $x+ -:y\longmapsto x+ y$
and $-( ):y\longmapsto -y$
are a complete $\vee$-morphism and a $\wedge$-morphism, respectively.
\end{thm}

\begin{defn}\label{L-subgroup}\cite{Mer}
Let $(G;+,0)$ be a group. By an $L$-{\it subgroup} of $(G;+,0)$
we mean an element $S\in L^G$ satisfying the following conditions:
\begin{itemize}
\item[(i)] $S(x)\wedge S(y)\leq S(x+y)$ for all $x,y\in G$;
\item[(ii)] $S(x)=S(-x)$ for all $x\in G$.
\end{itemize}
An $L$-subgroup $S$ of $(G;+,0)$ is called {\it normal} if $S(y)\leq S(x+y-x)$ for all
$x,y\in G$. Clearly, if $S$ is normal, then $S(y)= S(x+y-x)$ for all $x,y\in G$. Moreover, if
$S_1,\ldots,S_n\in L^G$, then $S_1+\cdots +S_n$ define by
$(S_1+\cdots +S_n)(y)=\vee\{S(x_1)\wedge S_(x_2)\wedge \cdots \wedge S(x_n)|\ x_1+\cdots+x_n=y\}$  for all $y\in G$.
\end{defn}

\begin{prop}{\rm\cite[Section 1]{Mer}}\label{Quo}
Let $S$ be a normal $L$-subgroup of a group $(G;+,0)$. Then the following conditions hold:
\begin{itemize}
\item[{\em (i)}] $S(x)\leq S(0)$ for all $x\in G$.
\item[{\em (ii)}] If $\alpha=S(0)$, then $S^{-1}(\alpha)$ is a normal subgroup $G$ and $S(x+y)=S(y+x)$ for all $x,y\in G$.
\item[{\em (iii)}] Consider the set $G+S=\{x+S|\ x\in G\}$. Define the binary operation $\oplus$ on $G+S$ by
$(x+S)\oplus (y+S)=(x+y)+S$. Then $(G+S;\oplus,S)$ is a group, where $-(x+S)=-x+S$ for all $x\in G$.
\item[{\em (iv)}] $(\frac{G}{S^{-1}(\alpha)};+,S^{-1}(\alpha))$ and $(G+S;\oplus,S)$ are isomorphic and so
$a+S^{-1}(\alpha)=b+S^{-1}(\alpha)$ if and only if $a+S=b+S$ for all $a,b\in G$
\end{itemize}
\end{prop}
\section{$L$-lattices}

In this section, we use the notions of join and meet on an $L$-ordered set and define an $L$-lattice.
In fact, this definition is a generalization of an $L$-complete lattice definition.
Then we find some results related to $L$-lattices which will be used in the next section.

From now on, in this paper,  $L^{P}_{\bullet }$ denotes the set of all elements
of $L^{P}$ with finite support, where support $S$ is the set $supp(S):=\{x\in P|\ 0<S(x)\}$.
\begin{defn}\label{3.1}
An $L$-ordered set $(P;e)$ is called an {\it $L$-lattice} if, for all $S\in L^{P}_{\bullet }$, $\sqcup S$ and
$\sqcap S$ exist.
\end{defn}
We must note that any complete $L$-lattice is an $L$-lattice.
\begin{prop}\label{3.2}
If $(P;e)$ is an $L$-lattice, then $(P;\leq_e)$ is a lattice.
\end{prop}

\begin{proof}
Let $x$ and $y$ be arbitrary elements of $P$. Then $\chi_{\{x,y\}}\in L^{P}_{\bullet }$.
Suppose that $\sqcup \chi_{\{x,y\}}=b$
and $\sqcap \chi_{\{x,y\}}=a$, for some $a,b\in P$.
By Theorem \ref{join and meet}, for every $v\in P$, we have
$$e(b,v)=\wedge_{u\in P} (\chi_{\{x,y\}}(u)\ra e(u,v))=1\ra e(x,v)\wedge 1\ra e(y,v)=e(x,v)\wedge e(y,v)$$
and so $1=e(b,b)=e(x,b)\wedge e(y,b)$. Thus, $b$ is an upper bound for $\{x,y\}$. Now, let $x\leq_e c$  and
$y\leq_e c$ for some $c\in P$. Then $e(b,c)=e(x,c)\wedge e(y,c)=1$. It follows that $b\leq_e c$, so
$b=x\vee y$, in the poset $(P;\leq_e)$. In a similar way, we can show that $a=x\wedge y$.
Therefore, $(P;\leq_e)$ is a lattice.
\end{proof}

\begin{defn}\label{3.3}
An $L$-lattice $(P;e)$ is called {\it distributive}  if, for all
$S\in L^{P}_{\bullet} $ and all $a\in P$, the following conditions hold:
\begin{eqnarray}
\label{dis} a\wedge \sqcup S=\sqcup(a\wedge S)\ \ \mbox{ and } \ \ a\vee \sqcap S=\sqcap (a\vee S),
\end{eqnarray}
where $(a\vee S)(y)=\vee\{S(x)|\ x\in P, \ a\vee x=y \}$ and $(a\wedge S)(y)=\vee\{S(x)|\ x\in P,\  a\wedge x=y \}$
for all $y\in P$.
\end{defn}

It is obvious that, if $(P;e)$ holds on one of the
conditions in Definition \ref{3.3}, then $(P;\leq_e)$ is a distributive lattice.
In the next remark, we want to verify some important properties of an $L$-lattice.

\begin{rmk}\label{3.4}
Let $(P;e)$ be an $L$-lattice and $S\in L^{P}_{\bullet}$.
Then there exist $x_1,\ldots,x_n\in P$ such that $supp(S)=\{x_1,\ldots,x_n\}$ for some $n\in \mathbb{N}$.

{\rm (i)}  Suppose that $\sqcap S=a$ and $\sqcup S=b$. Then, for all $x\in P$, by Theorem \ref{join and meet},
\begin{eqnarray*}
e(b,x)&=&\wedge_{y\in P}( S(y)\ra e(y,x))\\
&=&\wedge_{i=1}^{n} (S(x_i)\ra e(x_i,x)), \mbox{ since $s(y)\ra e(y\ra x)=1$ for all $y\in P-\{x_1,\ldots,x_n\}$ }
\end{eqnarray*}
In a similar way, we have $e(x,a)=\wedge_{i=1}^{n} (S(x_i)\ra e(x,x_i))$ for all $x\in P$.

{\rm (ii)} By Proposition \ref{3.2}, $\wedge^{n}_{i=1} x_i$ and $\vee^{n}_{i=1} x_i$ exist in the lattice
$(P;\leq_e)$. Also, by (i),
\begin{eqnarray*}
e(\wedge^{n}_{i=1} x_i,a)=\wedge_{i=1}^{n} (S(x_i)\ra e(\wedge^{n}_{i=1} x_i,x_i))=\wedge_{i=1}^{n}(S(x_i)\ra 1)=1\\
e(b,\vee^{n}_{i=1} x_i)= \wedge_{i=1}^{n} (S(x_i)\ra e(x_i,\vee^{n}_{i=1} x_i))=\wedge_{i=1}^{n}(S(x_i)\ra 1)=1,
\end{eqnarray*}
so, $\wedge^{n}_{i=1} x_i\leq_e a$ and $b\leq_e \vee^{n}_{i=1} x_i$.

{\rm (iii)} Suppose that $S_1:P\ra L$ and $\hat{S_1}:P\ra L$ are defined by
\begin{equation*}
S_1(x)=\left\{\begin{array}{ll}
S(x) & \text{ $x=x_1$},\\
0 & \text{ otherwise}.\\
\end{array} \right.
\hspace{1cm} \hat{S_1}(x)=\left\{\begin{array}{ll}
S(x) & \text{if $x\in \{x_2,x_3,\ldots,x_n\}$},\\
0 & \text{ otherwise}.\\
\end{array} \right.
\end{equation*}
Then $S_1,\hat{S_1}\in L^{P}_{\bullet}$. Let $\sqcap S_1=a_1$ and $\sqcap \hat{S_1}=\hat{a_1}$, for
some $a_1,\hat{a_1}\in P$. We claim that $a=a_1\wedge \hat{a_1}$. By $\sqcap S=a$,
and (i), we obtain that $e(x,a)=\wedge_{i=1}^{n} (S(x_i)\ra e(x,x_i))$, for
all $x\in P$. It follows that $e(a_1\wedge \hat{a_1},a)= \wedge_{i=1}^{n} (S(x_i)\ra e(a_1\wedge \hat{a_1},x_i))$.
Since $\sqcap S_1=a_1$ and $\sqcap \hat{S_1}=\hat{a_1}$, then by (i),
$S(x_1)\ra e(a_1\wedge \hat{a_1},x_1)=e(a_1\wedge \hat{a_1},a_1)=1$ and
$1=e(a_1\wedge \hat{a_1},\hat{a_1})= \wedge_{i=2}^{n} (S(x_i)\ra e(a_1\wedge \hat{a_1},x_i))$. Hence,
$e(a_1\wedge \hat{a_1},a)=1$ whence $a_1\wedge \hat{a_1}\leq a$. Also, from $\sqcap S_1=a_1$,
$\sqcap \hat{S_1}=\hat{a_1}$ and (i), we conclude that
$e(a,a_1)=S(x_1)\ra e(a,x_1)$ and $e(a,\hat{a_1})=\wedge_{i=2}^{n} (S(x_i)\ra e(a,x_i))$, so
$$e(a,a_1)\wedge e(a,\hat{a_1})=(S(x_1)\ra e(a,x_1))\wedge (\wedge_{i=2}^{n} (S(x_i)\ra e(a,x_i)))=\wedge_{i=1}^{n} (S(x_i)\ra e(a,x_i))=e(a,a)=1.$$
Thus, $a\leq a_1$ and $a\leq \hat{a_1}$. Therefore, $a=a_1\wedge \hat{a_1}$.
By a similar way, we can show that if $\sqcup S=b$, $\sqcup S_1=b_1$ and $\sqcup \hat{S_1}=\hat{b_1}$, for
some $b,b_1,\hat{b_1}\in P$, then $b=b_1\vee \hat{b_1}$.

{\rm (iv)} Let $a,b,c\in P$ such that $a\leq_e b$. Then by {\rm (E2)}, we get that
$$e(c,a)=e(c,a\wedge b)=e(c,a\wedge b)\wedge e(a\wedge b,b)\leq e(c,b).$$
\end{rmk}

\begin{prop}\label{3.4.2}
Let $(P;e)$ be an $L$-lattice and $S\in L^{P}_{\bullet}$. Then the following conditions hold:
\begin{itemize}
\item[{\rm (i)}] $e(a,x\wedge y)=e(a,x)\wedge e(a,y)$ for all $x,y,a\in P$;
\item[{\rm (ii)}] $e(x\vee y,a)=e(x,a)\wedge e(y,a)$ for all $x,y,a\in P$;
\item[{\rm (iii)}]  if $\vee_{y\in P}S(y)=1$, then $a\wedge \sqcap S=\sqcap (a\wedge S)$;
\item[{\rm (iv)}]  if $\vee_{y\in P}S(y)=1$, then $a\vee \sqcup S=\sqcup (a\vee S)$.
\end{itemize}
\end{prop}

\begin{proof}
(i) Put $x,y,a\in P$. Define $T:P\ra L$ by $T=\chi_{\{x,y\}}$. Then by the assumptions, $\sqcap T$ exists.
Let $m=\sqcap T$. By Theorem \ref{join and meet}(ii) and Remark \ref{3.4}(i),  for all $t\in P$,
$e(t,m)=(T(x)\ra e(t,x))\wedge (T(y)\ra e(t,y))=e(t,x)\wedge e(t,y)$. By Remark \ref{3.4}(ii),
we know that $x\wedge y\leq m$. Also, $1=e(m,m)=e(m,x)\wedge e(m,y)$, so $m\leq x$ and $m\leq y$.
It follows that, $m=x\wedge y$. Hence, for all $t\in P$,
$e(t,x\wedge y)=e(t,x)\wedge e(t,y)$. Since $x$ and $y$ are arbitrary elements of $P$, we get  $e(a,x\wedge y)=e(a,x)\wedge e(a,y)$
for all $x,y,a\in P$.

(ii) The proof of this part is similar to (i).

(iii) Let $Supp(S)=\{x_1,x_2,\ldots,x_n\}$ and $\sqcap S=m$ for some $m\in P$. Then by Remark \ref{3.4}(i),
$e(x,m)=\wedge_{i=1}^{n} (S(x_i)\ra e(x,x_i))$ for all $x\in P$. Let $T=a\wedge S$. It suffices to show that $\sqcap T=a\wedge m$.
For all $y\in P$, $T(y)=\vee_{\{t\in P|\ a\wedge t=y\}}S(t)$. Since $Supp(S)=\{x_1,x_2,\ldots,x_n\}$,
then $Supp(T)=\{a\wedge x_1,a\wedge x_2,\ldots,a\wedge x_n\}$, thus $T\in L^{P}_{\bullet}$ and $\sqcap T$ exists. Put $x\in P$.
$\wedge_{y\in P}(T(y)\ra e(x,y))=\wedge_{y\in P}((\vee_{\{t\in P|\ a\wedge t=y\}}S(t))\ra e(x,y))=\wedge_{y\in P}\wedge_{a\wedge t=y}(S(t)\ra e(x,y))=
\wedge_{t\in P}(S(t)\ra e(x,a\wedge t))$. Since $Supp(S)=\{x_1,x_2,\ldots,x_n\}$, we get that
\begin{eqnarray*}
\wedge_{y\in P}(T(y)\ra e(x,y))&=&\wedge_{i=1}^{n}(S(x_i)\ra e(x,a\wedge x_i))\\
&=&\wedge_{i=1}^{n}(S(x_i)\ra (e(x,a)\wedge e(x,x_i))), \mbox{ by (i)},\\
&=& (\wedge_{i=1}^{n}(S(x_i)\ra e(x,a)))\wedge (\wedge_{i=1}^{n}(S(x_i)\ra e(x,x_i)))\\
&=& (\wedge_{i=1}^{n}(S(x_i)\ra e(x,a)))\wedge e(x,m), \mbox{ since $\sqcap S=m$}, \\
&=& [(\vee_{i=1}^{n}S(x_i))\ra e(x,a)]\wedge e(x,m)\\
&=& e(x,a)\wedge e(x,m), \mbox{ since $\vee_{y\in P}S(y)=1$ }\\
&=&e(x,a\wedge m),  \mbox{ by (i)}.
\end{eqnarray*}
Therefore, by Theorem \ref{join and meet}(ii), $\sqcap T=a\wedge m$.

(iv) The proof of this part is similar to (iii).
\end{proof}

\begin{cor}
Let $(P;e)$ be an $L$-lattice and $S\in L^{P}_{\bullet}$ such that $\vee_{y\in P}S(y)=1$. Then
$a\wedge \sqcup(a\vee S)=a$ and $a\vee \sqcap(a\wedge S)=a$ for all $a\in P$.
\end{cor}

\begin{prop}\label{semilattice}
Let $(P;e)$ be an $L$-lattice and $S,T\in L^{P}$ such that $\sqcup S$, $\sqcap S$, $\sqcup T$ and $\sqcap T$ exist.
Consider the maps $S\cup T:P\ra L$ and $S\cap T:P\ra L$, which are defined by $(S\cup T)(x)=S(x)\vee T(x)$ and
$(S\cap T)(x)=S(x)\wedge T(x)$, respectively. Then $\sqcup (S\cup T)=\sqcup S\vee \sqcup T$ and $\sqcap (S\cup T)=\sqcap S\wedge \sqcap T$.
\end{prop}

\begin{proof}
Let $\sqcup S=J_1$ and  $\sqcup T=J_2$.
For each $x\in P$, $e(J_1,x)=\wedge_{y\in P}(S(y)\ra e(y,x))$ and
$e(J_2,x)=\wedge_{y\in P}(T(y)\ra e(y,x))$. Put $x\in P$. We have
\begin{eqnarray*}
\wedge_{y\in P}((S\cup T)(y)\ra e(y,x))&=& \wedge_{y\in P}((S(y)\vee T(y))\ra e(y,x))\\
&=& \wedge_{y\in P}([S(y)\ra e(y,x)]\wedge [T(y)\ra e(y,x)])\\
&=& (\wedge_{y\in P}[S(y)\ra e(y,x)])\wedge (\wedge_{y\in P}[T(y)\ra e(y,x)])\\
&=& e(J_1,x)\wedge e(J_2,x)\\
&=& e(J_1\vee J_2,x),\mbox{ by Proposition \ref{3.4.2}(ii),}
\end{eqnarray*}
so, Theorem \ref{join and meet} implies that $\sqcup (S\cup T)=J_1\vee J_2$.
Moreover, for all $x\in P$, $e(x,m_1)=\wedge_{y\in P}(S(y)\ra e(x,y))$ and
$e(x,m_2)=\wedge_{y\in P}(T(y)\ra e(x,y))$. For all $x\in P$, we have
\begin{eqnarray*}
\wedge_{y\in P}((S\cup T)(y)\ra e(x,y))&=& \wedge_{y\in P}((S(y)\vee T(y))\ra e(x,y))\\
&=& \wedge_{y\in P}([S(y)\ra e(x,y)]\wedge [T(y)\ra e(x,y)])\\
&=& (\wedge_{y\in P}[S(y)\ra e(x,y)])\wedge (\wedge_{y\in P}[T(y)\ra e(x,y)])\\
&=& e(x,m_1)\wedge e(x,m_2)\\
&=& e(x,m_1\wedge m_2),\mbox{ by Proposition \ref{3.4.2}(i)}.
\end{eqnarray*}
Hence, by Theorem \ref{join and meet}, $\sqcap (S\cup T)=m_1\wedge m_2$.
\end{proof}

It follows from Proposition \ref{semilattice} that  if $A=\{S\in L^{P} |\ \sqcap S\mbox{ and } \sqcup S \mbox{ exist}\}$, then
$S\cup T\in A$. It can be easily shown that $(A;\cup)$ is an upper semilattice.
It follows that $(L^{P}_{\bullet};\cup)$ is an upper semilattice, too.

\section{$L$-ordered groups} 

In this section, the concept of an $L$-ordered group is introduced; $L$ stands for a fixed bounded lattice $L$. It is an
group $(G;+,0)$ written additively
equipped with an $L$-order relation $e:G\times G\ra L$ such that the transition map $( \ )+a:G\ra G$
and $a+( \ ):G\ra G$ are monotone for all $a\in G$.
Then we give some useful properties of an $L$-ordered group and show that
each transition maps $( \ )+a:G\ra G$ and $a+( \ ):G\ra G$ preserve both join and meet $S$
for all $S\in L^G$ (if join and meet
$S$ exist). Then the notion of a convex $L$ subgroup is defined and it is used to construct
quotient $L$-groups.
Finally, a relation between a positive cone of an $L$-ordered group and special type of
elements of $L^G$ is obtained.

\begin{defn}\label{4.1}
An {\it $L$-fuzzy ordered group} (an $L$-ordered group for short) is a group $(G;+,0)$
together with an $L$-order relation $e:G\times G\ra L$ on $G$ such that each translation maps
$( \ )+a:G\ra G$ and $a+( \ ):G\ra G$ are monotone, or equivalently: \\

${\rm(FOG)}$ for each $x,y,a,b\in G$, $e(x,y)\leq e(b+x+a,b+y+a)$. \\

We use $(G;e,+,0)$ to denote $G$ is an $L$-fuzzy ordered group. An {\it $L$-lattice ordered group} is the structure $(G;e,+,0)$,
where $(G;e,+,0)$ is an $L$-ordered group and  $(G;e)$ is an $L$-lattice.
In fact, an $L$-fuzzy ordered group $(G;e,+,0)$
is an $L$-lattice ordered group if for each $S\in L^{G}_{\bullet }$, $\sqcup S$ and $\sqcap S$ exist.
\end{defn}

In this section, after we find some results about $L$-ordered groups, we will propose a procedure to construct $L$-ordered groups (see Remark \ref{rmk}).

\begin{rmk}\label{4.2}
(i) Let $G$ be a lattice ordered group. Consider the $L$-ordered set $(G;\chi_\leq)$.
It is easy to show that $L$-ordered set $(G;\chi_\leq)$ is a
$\{0,1\}$-lattice ordered group. Moreover, if $(G;e,+,0)$ is an $L$-lattice ordered group, then clearly, $(G;+,0,\leq_e)$ is a crisp one. Therefore, by \cite[Proposition 1.1.7]{anderson-feil}, if $(G;e,+,0)$ is an
$L$-ordered group, then it is torsion free.

(ii) Let $(P,e)$ be a $L$-ordered set and $A(P)$ be the set of all automorphisms of an $L$-ordered set $P$,
that is $f:P\ra P$ and $f^{-1}:P\ra P$ are monotone. Define $\hat{e}:A(P)\times A(P)\ra L$ by
$\hat{e}(f,g)=\wedge_{x\in P}e(f(x),g(x))$
for all $f,g\in A(P)$. Clearly, $(A(P);\hat{e})$ is an $L$-ordered set. Put $f,g,\alpha\in A(P)$.
Since $\alpha$ is monotone, then $e(f(x),g(x))\leq e(\alpha(f(x)),\alpha(g(x)))$ for all $x\in P$, hence
$\hat{e}(f,g)=\wedge_{x\in P}e(f(x),g(x))\leq \wedge_{x\in P}e(\alpha(f(x)),\alpha(g(x)))=\hat{e}(\alpha\circ f,\alpha\circ g)$.
Also, $\wedge_{x\in P}e(f(x),g(x))\leq \wedge_{x\in P}e(f(\alpha(x)),g(\alpha(x)))$ (since $Im(\alpha)\s P$), so
$(A(P);\hat{e},\circ,Id_P)$ is an $L$-ordered group.
\end{rmk}

In the next proposition, we will gather some useful properties of $L$-ordered groups.

\begin{prop}\label{4.3}
Let $(G;e,+,0)$ be an $L$-ordered group. Then the following conditions hold:
\begin{itemize}
\item[{\rm (i)}] For each $x,y,a,b\in G$, $e(x,y)=e(b+x+a,b+y+a)$.
\item[{\rm (ii)}] For each $x,y\in G$, $e(x,y)=e(-y,-x)$.
\item[{\rm (iii)}] If $S\in L^G$ such that $\sqcup S$ exists, then for each
$a\in G$, $a+\sqcup S=\sqcup (a+ S)$ and $\sqcup S+a=\sqcup (S+a)$.
\item[{\rm (iv)}] If $S\in L^G$ such that $\sqcap S$ exists, then for each $a\in G$,
$a+ \sqcap S=\sqcap (a+ S)$ and $\sqcap S+a=\sqcap (S+a)$.
\item[{\rm (v)}] If $S\in L^G$ such that $\sqcup S$ $(\sqcap S)$  exists, then $-(\sqcup S)=(\sqcap -S)$ $(-(\sqcap S)=(\sqcup -S))$.
\item[{\rm (vi)}] If $x,y,a\in G$ such that $x\leq y$, then $e(y,a)\leq e(x,a)$.
\item[{\rm (vii)}] For all $a\in G$, the maps $a\wedge -:G\ra G$ and $a\vee -:G\ra G$ are monotone.
\end{itemize}
\end{prop}

\begin{proof}
(i) For all $x,y,a\in G$, by ${\rm(FOG)}$ we have $e(x,y)\leq e(b+x+a,b+y+a)$ and $e(b+x+a,b+y+a)\leq e(-b+(b+x+a)-a,-b+(b+y+a)-a)=e(x,y)$. So, $e(x,y)=e(b+x+a,b+y+a)$.

(ii) Let $x,y\in G$. Then by (i), $e(x,y)=e(x-x,y-x)=e(0,y-x)=e(-y,-y+y-x)=e(-y,-x)$.

(iii) Let $S\in L^G$, $\sqcup S=J$ and $a\in G$.  We show that $a+ J=\sqcup (a+ S)$.
By Theorem \ref{join and meet}(i), it
suffices to show that $e(a+J,x)=\wedge_{y\in G} ((a+ S)(y)\ra e(y,x))$ for all $x\in G$. Take $x\in G$.
\begin{eqnarray*}
\wedge_{y\in P} ((a+S)(y)\ra e(y,x))&=& \wedge_{y\in G}[(\vee_{\{t\in G| \ a+ t=y\}} S(t))\ra e(y,x)]\\
&=& \wedge_{y\in G}\wedge_{\{t\in G| \ a+t=y\}} (S(t)\ra e(y,x))\\
&=& \wedge_{t\in G} (S(t)\ra e(a+ t,x))\\
&=& \wedge_{t\in G} (S(t)\ra e(-a+a+t,-a+x)), \mbox{ by (i),}\\
&=& \wedge_{t\in G} (S(t)\ra e(t,-a+x)).
\end{eqnarray*}
Since $\sqcup S=J$, then by Theorem \ref{join and meet}(i),
$e(J,-a+x)=\wedge_{y\in G} (S(y)\ra e(y,-a+x))$. Hence by (i),
we obtain that $\wedge_{y\in P} ((a+S)(y)\ra e(y,x))=e(J,-a+x)=e(a+J,x)$. Therefore, $a+ J=\sqcup (a+ S)$.
The proof of the other part is similar.

(iv) The proof of this part is similar to (iii).

(v) Let $m=\sqcap S$. Then for each $x\in G$, $e(x,m)=\wedge_{y\in G} (S(y)\ra e(x,y))$. For all $x\in G$, by (ii),
$$e(-m,-x)=\wedge_{y\in G} (S(y)\ra e(x,y))=\wedge_{y\in G} (S(y)\ra e(-y,-x)).$$
So, for all $u\in G$, we have $e(-m,u)=\wedge_{y\in G} (S(y)\ra e(-y,u))$. It follows that,
$e(-m,u)=\wedge_{y\in G} (S(-y)\ra e(y,u))=\wedge_{y\in G} ((-S)(y)\ra e(y,u))$ for all $u\in G$.
Therefore, by Theorem \ref{join and meet}(ii),
we conclude that $(\sqcap -S)=-m=-(\sqcap S)$. The proof of the rest is similar.

(vi) Let $x,y,a\in G$ and $x\leq_e y$. Then $-y\leq_e -x$, so by Remark \ref{3.4}(iv), $e(-a,-y)\leq e(-a,-x)$.
Hence by (ii), we get that $e(y,a)\leq e(x,a)$.

(vii) By Proposition \ref{3.4.2}(iii),
$e(a\wedge x,y)\leq e(a\wedge x, a\wedge y)$ for all $x,y,a\in G$.
On the other hand, $a\wedge x\leq x$ and so by (vi),
$e(x,y)\leq e(a\wedge x,y)$, which implies that $e(x,y)\leq e(a\wedge x,a\wedge y)$. Therefore, $a\wedge -$ is monotone.
In a similar way, we can show that  $a\vee -$ is monotone.
\end{proof}

Proposition \ref{4.3}(v) entails the following corollary:

\begin{cor}\label{lattice}
An $L$-ordered group $(G;e,+,0)$ is an $L$-lattice ordered group if and only if
$\sqcup S$ exists for any $S\in L^{G}_{\bullet }$ or equivalently,
$\sqcap S$ exists for any $S\in L^{G}_{\bullet }$.
\end{cor}

\begin{cor}\label{4.4}
Suppose that $(G;e,+,0)$ is an $L$-lattice ordered group, then for any $S\in L^{G}_{\bullet }$ and $a\in G$, we have
\begin{itemize}
\item[{\em (i)}] $a+ \sqcup S=\sqcup (a+ S)$, $a+ \sqcap S=\sqcap (a+ S)$ $-(\sqcup S)=(\sqcap -S)$ and $-(\sqcap S)=(\sqcup -S)$.
\item[{\em (ii)}] If $(G;e)$ is complete, then for each $a\in G$, the transition map $( \ )+a:G\ra G$ has both left and right adjoint.
\end{itemize}
\end{cor}

\begin{proof}\label{4.5}
(i) It follows from Proposition \ref{4.3}(iii)--(v).

(ii) It follows from Proposition \ref{4.3}(iii), (iv) and Theorem \ref{adjoint}.
\end{proof}

\begin{prop}\label{4.10}
Let $(G;e,+,0)$ be an $L$-lattice ordered group and $S,T\in L^{G}_{\bullet }$. Then $\sqcup (S+T)=\sqcup S+\sqcup T$ and
$\sqcap (S+T)=\sqcap S+\sqcap T$.
\end{prop}

\begin{proof}
Suppose that $j_1=\sqcup S$ and $j_2=\sqcup T$. Then by Proposition \ref{join and meet}, for each $x\in G$, $e(j_1,x)=\wedge_{y\in G}(S(y)\ra e(y,x))$ and
$e(j_2,x)=\wedge_{y\in G}(T(y)\ra e(y,x))$. Take $x\in G$.
\begin{eqnarray*}
\wedge_{y\in G}((S+T)(y)\ra e(y,x))&=& \wedge_{y\in G}((\vee_{y_1+y_2=y}(S(y_1)\wedge T(y_2)))\ra e(y,x))\\
&=& \wedge_{y\in G}\wedge_{y_1+y_2=y}((S(y_1)\wedge T(y_2))\ra e(y,x))\\
&=& \wedge_{y\in G}\wedge_{y_1+y_2=y}(S(y_1)\ra (T(y_2)\ra e(y,x)))\\
&=& \wedge_{y_1,y_2\in G}(S(y_1)\ra (T(y_2)\ra e(y_1+y_2,x)))\\
&=& \wedge_{y_1\in G}\wedge_{y_2\in G}(S(y_1)\ra (T(y_2)\ra e(y_2,-y_2+x)))\\
&=& \wedge_{y_1\in G}(S(y_1)\ra (\wedge_{y_2\in G}(T(y_2)\ra e(y_2,-y_2+x))))\\
&=& \wedge_{y_1\in G}(S(y_1)\ra e(j_2,-y_1+x))\\
&=& \wedge_{y_1\in G}(S(y_1)\ra e(y_1,x-j_2))\\
&=&e(j_1,x-j_2)=e(j_1+j_2,x).
\end{eqnarray*}
Hence $\sqcup (S+T)=\sqcup S+\sqcup T$. By a similar way, we can show that $\sqcap (S+T)=\sqcap S+\sqcap T$.
\end{proof}

\begin{defn}\label{convex}
Let $(P;e)$ be an $L$-ordered set. An element $S\in L^P$ is called {\it convex} if for every $x,y,a\in P$,
$S(a)\geq S(x)\wedge S(y)\wedge e(x,a)\wedge e(a,y)$.
\end{defn}

\begin{exm}\label{make-convex}
Let $(P;e)$ be an $L$-ordered set and $S\in L^P$. Define $\overline{S}:P\ra L$ by
$\overline{S}(a)=\vee_{x,y\in P}(S(x)\wedge S(y)\wedge e(x,a)\wedge e(a,y))$. We claim that $\overline{S}$ is convex. In fact, for each
$x,y,a\in G$, $\overline{S}(x)=\vee_{t_1,t_2\in G}(S(t_1)\wedge S(t_2)\wedge e(t_1,x)\wedge e(x,t_2))$ and
$\overline{S}(y)=\vee_{u_1,u_2\in G}(S(u_1)\wedge S(u_2)\wedge e(u_1,y)\wedge e(y,u_2))$, so, by
\begin{eqnarray*}
\overline{S}(x)\!\!\!\!&\wedge&\!\!\!\! \overline{S}(y)\wedge e(x,a)\wedge e(a,y)\\
\!\!\!\!&=&\!\!\!\! \vee_{t_1,t_2,u_1,u_2\in P}
(S(t_1)\!\wedge\! S(t_2)\!\wedge\! e(t_1,x)\!\wedge\! e(x,t_2)\!\wedge\! S(u_1)\!\wedge\! S(u_2)\!\wedge\! e(u_1,y)\!\wedge\! e(y,u_2)\!\wedge\! e(x,a)\!\wedge\! e(a,y))\\
\!\!\!\!&\leq&\!\!\!\! \vee_{t_1,t_2,u_1,u_2\in P}(S(t_1)\wedge S(t_2)\wedge e(t_1,a)\wedge e(x,t_2)\wedge S(u_1)\wedge S(u_2)\wedge e(u_1,y)\wedge e(a,u_2))\\
\!\!\!\!&\leq&\!\!\!\! \vee_{t_1,t_2,u_1,u_2\in P}(S(t_1)\wedge S(u_2)\wedge e(t_1,a)\wedge e(a,u_2))\\
\!\!\!\!&=&\!\!\!\! \overline{S}(a).
\end{eqnarray*}
we get that $\overline{S}$ is convex. Moreover, it is clear that if $T\in L^P$ is convex such that $S(x)\leq T(x)$, for all
$x\in P$, then $\overline{S}\s T$.
\end{exm}

Note that if $(P;e)$ is an $L$-ordered set, then $S\in L^G$ is convex if and only if
$S(a)=\vee_{x\in P}\vee_{y\in P} [S(x)\wedge S(y)\wedge e(x,a)\wedge e(a,y)]$ for all $a\in G$. First, we assume that $S$ is convex and
$a\in P$, then $S(a)$ is an upper bound for the set $\{S(x)\wedge S(y)\wedge e(x,a)\wedge e(a,y)|\ x,y\in P\}$. Also,
$S(a)=S(a)\wedge S(a)\wedge e(a,a)\wedge e(a,a)$ is an element of this set, so
$S(a)=\vee_{x\in P}\vee_{y\in P} [S(x)\wedge S(y)\wedge e(x,a)\wedge e(a,y)]$. The proof of the converse is clear.

\begin{prop}\label{L-subgroup prop}
Let $S$ be a normal $L$-subgroup of an $L$-ordered group $(G;e,+,0)$. Then
$S$ is convex if and only if $S(a)\geq S(x)\wedge e(0,a)\wedge e(a,x)$ for all $x,a\in G$.
\end{prop}

\begin{proof}
Let $S(a)\geq S(x)\wedge e(0,a)\wedge e(a,x)$ for all $x,a\in G$. Put $x,y,a\in G$. By the assumption,
$S(a-x)\geq S(y-x)\wedge e(0,a-x)\wedge e(a-x,y-x)$. Since  $(G;e,+,0)$ is an $L$-ordered
group and $S$ is its $L$-subgroup, we get
$S(y-x)\geq S(y)\wedge S(x)$, $e(0,a-x)=e(x,a)$ and $e(a-x,y-x)=e(a,y)$, so
$S(a)=S(a-x+x)\geq S(a-x)\wedge S(x)\geq
S(y-x)\wedge e(0,a-x)\wedge e(a-x,y-x)\wedge S(x)
\geq S(x)\wedge S(y)\wedge e(x,a)\wedge e(a,y)$. Therefore,
$S$ is convex. The proof of the converse is clear.
\end{proof}

\begin{defn}
Let $(G;e,+,0)$ be an $L$-ordered group and $S\in L^G$. By a {\it positive cone} of $S$, we mean a map
$S^+:G\ra L$ defined by $S^+(x)=S(x)\wedge e(0,x)$ for all $x\in G$. The positive cone of $G$ is the positive
cone of the map $S_G:G\ra L$ sending each element $x\in G$ to $e(0,x)$.
So, for each $x\in G$, $(S_G)^{+}(x)=e(0,x)\wedge e(0,x)=e(0,x)$.
\end{defn}

Let $(G;e,+,0)$ be an $L$-ordered group. Then there exists an interesting relation between the positive cone of
$S$ and $\downarrow S$ for an $L$-subset $S\in L^G$.

Assume that $S$ is a convex $L$-subgroup of $(G;e,+,0)$ and $a\in G$. For each $x\in G$,
$S(x)\wedge e(0,a)\wedge e(a,x)=S(x)\wedge S(0)\wedge e(0,a)\wedge e(a,x)\leq S(a)$, so
$e(0,a)\wedge \downarrow S(a)=e(0,a)\wedge \vee_{x\in G}(S(x)\wedge e(a,x))=\vee_{x\in G}(S(x)\wedge e(0,a)\wedge e(a,x)) \leq S(a)$.
It is easy to see that $e(0,a)\wedge \vee_{x\in G}(S(x)\wedge e(a,x))=e(0,a)\wedge S(a)$, hence
$(\downarrow S)^{+}(a)=e(0,a)\wedge \downarrow S(a)=S(a)\wedge e(0,a)=S^{+}(a)$, whence
$(\downarrow S)^{+}=S^{+}$. Moreover, if $T$ is an $L$-subgroup of $(G;e,+,0)$ such that $(\downarrow T)^{+}=T^{+}$,
then for each $a\in G$, $\vee_{x\in G}(T(x)\wedge e(a,x)\wedge e(0,a))\leq T(a)\wedge e(0,a)\leq T(a)$, and so
by Proposition \ref{L-subgroup prop}, $T$ is convex.

\begin{prop}\label{cone}
Let $(G;e,+,0)$ and $(H;e,+,0)$ be two $L$-ordered groups and $f:G\ra H$ be a group homomorphism.
Then the following are equivalent:
\begin{itemize}
\item[{\rm (i)}] $e(0,x)\leq e(0,f(x))$ for all $x\in G$.
\item[{\rm (ii)}] $f$ is monotone.
\item[{\rm (iii)}] $f(S_G)\s S_H$ (that is $f(S_G)(y)\leq S_H(y)$ for all $y\in H$).
\end{itemize}
\end{prop}

\begin{proof}
${\rm (i)\Rightarrow (ii)}$  Put $x,y\in G$. By (i) and Proposition \ref{3.4}(i), we have
$e(x,y)=e(0,y-x)\leq e(0,f(y-x))=e(0,f(y)-f(x))=e(f(x),f(y))$. Hence, $f$ is monotone.

${\rm (ii)\Rightarrow (iii)}$ Let $y\in H$. Then $f(S_G)(y)=\vee_{\{x\in G|\ f(x)=y\}}S_G(x)=\vee_{\{x\in G|\ f(x)=y\}}e(0,x)$. Since $f$ is monotone, then
$e(0,x)\leq e(0,f(x))=e(0,y)$ for all $x\in \{x\in G|\ f(x)=y\}$. It follows that $\vee_{\{x\in G|\ f(x)=y\}}S_G(x)\leq e(0,y)=S_H(y)$.

${\rm (iii)\Rightarrow (i)}$ For any $x\in G$, $e(0,x)\leq \vee_{\{a\in G|\ f(a)=f(x)\}}e(0,a)=f(S_G)(f(x))\leq S_H(f(x))=e(0,f(x))$.
\end{proof}

\begin{prop}\label{convex-subgroup}
Let $S$ be a normal $L$-subgroup of an $L$-ordered group $(G;e,+,0)$.
Then $\overline{S}$ is a normal convex $L$-subgroup.
\end{prop}

\begin{proof}
By Example \ref{make-convex}, we know that $\overline{S}$ is convex. Choose $a,b\in G$.
\begin{eqnarray*}
\overline{S}(a)\wedge \overline{S}(b)&=& \vee_{x,x'\in G}(S(x)\wedge S(x')\wedge e(x,a)\wedge e(a,x'))\wedge
\vee_{y,y'\in G}(S(y)\wedge S(y')\wedge e(y,b)\wedge e(b,y'))\\
&=&\vee_{x,y,x',y'\in G}(S(x)\wedge S(y)\wedge e(x,a)\wedge e(y,b)\wedge S(x')\wedge S(y')\wedge e(a,x') \wedge e(b,y')).
\end{eqnarray*}
By ${\rm(E2)}$, for each $u,v\in G$, $e(0,u+v)=e(-v,u)\geq e(-v,0)\wedge e(0,u)=e(0,v)\wedge e(0,u)$, so by Proposition \ref{4.3}(i),
$e(x,a)\wedge e(y,b)=e(0,-x+a)\wedge e(0,b-y)\leq e(0,-x+a+b-y)=e(x+y,a+b)$. In a similar way, $e(a,x') \wedge e(b,y')\leq
e(a+b,x'+y')$, whence $\overline{S}(a)\wedge \overline{S}(b)\leq
\vee_{x,y,x',y'\in G}(S(x)\wedge S(y)\wedge e(x+y,a+b)\wedge S(x')\wedge S(y')\wedge e(a+b,x'+y'))$. By the assumption, $S(x)\wedge S(y)\leq S(x+y)$
and $S(x')\wedge S(y')\leq S(x'+y')$, hence
\begin{eqnarray*}
\overline{S}(a)\wedge \overline{S}(b)&\leq& \vee_{x,y,x',y'\in G}(S(x+y)\wedge e(x+y,a+b)\wedge S(x'+y')\wedge e(a+b,x'+y'))\\
&=& \vee_{u,v\in G}(S(u)\wedge e(u,a+b)\wedge S(v)\wedge e(a+b,v)), \mbox{ since $G$ is a group, }\\
&=& \overline{S}(a+b).
\end{eqnarray*}
Also by Propositione \ref{4.3}(i) we have, $S(-a)=\vee_{x,x'\in G}(S(x)\wedge S(x')\wedge e(x,-a)\wedge e(-a,x'))=
\vee_{x,x'\in G}(S(x)\wedge S(x')\wedge e(a,-x)\wedge e(-x',a))$. Since $G$ is a group
$\vee_{x,x'\in G}(S(x)\wedge S(x')\wedge e(a,-x)\wedge e(-x',a))=
\vee_{x,x'\in G}(S(-x)\wedge S(-x')\wedge e(a,x)\wedge e(x',a))=
\vee_{x,x'\in G}(S(x)\wedge S(x')\wedge e(a,x)\wedge e(x',a))=\overline{S}(a)$.
Therefore, $\overline{S}$ is an $L$-subgroup of $(G;e,+,0)$.
\end{proof}

\begin{defn}
A normal convex $L$-subgroup of an $L$-ordered group $(G;e,+,0)$ is called an {\it $L$-filter}.
\end{defn}

\begin{thm}\label{quotient}
Let $S$ be an $L$-filter of an $L$-ordered group $(G;e,+,0)$ and $S(0)=\alpha$.
\begin{itemize}
\item[{\em (i)}] $(G+S;\tilde{e},\oplus,S)$ is an $L$-ordered group,
where the map $\tilde{e}:G+S\times G+S\ra L$ sending
every element $(a+S,b+S)$ of $G+S\times G+S$ to $\vee_{x\in S^{-1}(\alpha)}e(a-b,x)$.
\item[{\em (ii)}] $\tilde{S}:G+S\ra L$, defined by $\tilde{S}(a+S)=S(a)$ for all $a\in G$, is an
$L$-filter of $(G+S;\tilde{e},\oplus,S)$ such that $(\tilde{S})^{-1}(\alpha)$ has
 exactly one element.
\end{itemize}
\end{thm}

\begin{proof}
(i) First, we show that $\tilde{e}$ is well defined. Let $a+S=a'+S$ and $b+S=b'+S$, for some $a,a',b,b'\in G$.
The $a+S^{-1}(\alpha)=a'+S^{-1}(\alpha)$ and $b+S^{-1}(\alpha)=b'+S^{-1}(\alpha)$, so
there exist $s,t\in S^{-1}(\alpha)$ such that $a-a'=s$ and $b-b'=t$. Take $x\in S^{-1}(\alpha)$.
Then by Proposition \ref{4.3}(i),
$e(a-b,x)=e(-s+a-b+t,-s+x+t)=e((-s+a)-(-t+b),-s+x+t)=e(a'-b',-s+x+t)$.
Since $S^{-1}(\alpha)$ is a subgroup of $(G;+,0)$, then $-s+x+t\in S^{-1}(\alpha)$. It follows that
$e(a-b,x)=e(a'-b',-s+x+t)\leq \vee_{y\in S^{-1}(\alpha)}e(a'-b',y)=\tilde{e}(a'+S,b'+S)$.
Since $x$ is an arbitrary element of $S^{-1}(\alpha)$, then we
get $\tilde{e}(a+S,b+S)\leq \tilde{e}(a'+S,b'+S)$. In a similar way, we
can show that $\tilde{e}(a'+S,b'+S)\leq \tilde{e}(a+S,b+S)$. So,
$\tilde{e}$ is well defined. Now, we show that
$(G+S;\tilde{e}),$ is an $L$-ordered set.

(1) Let $a\in G$. $\tilde{e}(a+S,a+S)=\vee_{x\in S^{-1}(\alpha)}e(a-a,x)=\vee_{x\in S^{-1}(\alpha)}e(0,x)$. From
$0\in S^{-1}(\alpha)$ it follows that $e(0,0)\leq \vee_{x\in S^{-1}(\alpha)}e(0,x)=\tilde{e}(a+S,a+S)$.

(2) Let $a,b,c\in G$. Then by Proposition \ref{4.3},
\begin{eqnarray*}
\tilde{e}(a+S,b+S)\wedge \tilde{e}(b+S,c+S)&=& \vee_{x\in S^{-1}(\alpha)}e(a-b,x) \ \wedge \  \vee_{t\in S^{-1}(\alpha)}e(b-c,t)\\
&=& \vee_{x\in S^{-1}(\alpha)}e(-x+a,b) \ \wedge \  \vee_{t\in S^{-1}(\alpha)}e(b,t+c)\\
&=& \vee_{t\in S^{-1}(\alpha)}[(\vee_{x\in S^{-1}(\alpha)}e(-x+a,b))\wedge e(b,t+c)]\\
&=& \vee_{t\in S^{-1}(\alpha)} \vee_{x\in S^{-1}(\alpha)} [e(-x+a,b)\wedge e(b,t+c)]\\
&\leq & \vee_{t\in S^{-1}(\alpha)} \vee_{x\in S^{-1}(\alpha)} e(-x+a,t+c)\\
&=& \vee_{t\in S^{-1}(\alpha)} \vee_{x\in S^{-1}(\alpha)} e(a-c,x+t)\\
&=&  \vee_{x,t\in S^{-1}(\alpha)} e(a-c,x+t).
\end{eqnarray*}
Since $S^{-1}(\alpha)$ is a subgroup of $G$, then $\{e(a-c,x+t)|\ x,t\in S^{-1}(\alpha) \}=\{e(a-c,x)|\ x\in S^{-1}(\alpha)\}$. It follows that
$\tilde{e}(a+S,b+S)\wedge \tilde{e}(b+S,c+S)\leq \vee_{x,t\in S^{-1}(\alpha)} e(a-c,x+t)=
\vee \{e(a-c,x)|\ x\in S^{-1}(\alpha)\}=\tilde{e}(a+S,c+S)$.

(3) Let $a,b\in G$ such that $\tilde{e}(a+S,b+S)=\tilde{e}(b+S,a+S)=1$. Then
$$\vee_{x\in S^{-1}(\alpha)}e(a-b,x)=1=\vee_{x\in S^{-1}(\alpha)}e(b-a,x).$$
Since $S$ is convex, then
$S(b-a)\geq S(x)\wedge S(t)\wedge e(x,a-b)\wedge e(a-b,t)=\alpha \wedge e(x,a-b)\wedge e(a-b,t)$
for all $x,t\in S^{-1}(\alpha)$. Thus,
\begin{eqnarray*}
S(a-b)&\geq& \vee_{x,t\in S^{-1}(\alpha)} [\alpha \wedge e(x,a-b)\wedge e(a-b,t)]\\
&=& \alpha \wedge (\vee_{x,t\in S^{-1}(\alpha)} [e(x,a-b)\wedge e(a-b,t)])\\
&=& \alpha\wedge  (\vee_{x\in S^{-1}(\alpha)} e(x,a-b) \ \wedge \ \vee_{t\in S^{-1}(\alpha)} e(a-b,t))\\
&=& \alpha\wedge (\vee_{x\in S^{-1}(\alpha)} e(b-a,-x) \ \wedge \ \vee_{t\in S^{-1}(\alpha)} e(a-b,t)).
\end{eqnarray*}
Since $S^{-1}(\alpha)$ is a subgroup of $G$, then $\vee_{x\in S^{-1}(\alpha)} e(b-a,-x)=\vee_{x\in S^{-1}(\alpha)} e(b-a,x)$ and so
\begin{eqnarray*}
S(a-b)&\geq& \alpha\wedge(\vee_{x\in S^{-1}(\alpha)} e(b-a,x) \ \wedge \ \vee_{t\in S^{-1}(\alpha)} e(a-b,t))\\
&=& \alpha \wedge 1\wedge 1=\alpha.
\end{eqnarray*}
Hence $a-b\in S^{-1}(\alpha)$, which implies that $a+S^{-1}(\alpha)=b+S^{-1}(\alpha)$ and
$a+S=b+S$.

Now, let $a,b,c\in G$.
$\tilde{e}((a+S)\oplus (c+S), (b+S)\oplus (c+S))=\tilde{e}((a+c+S), (b+c+S))=\vee_{x\in S^{-1}(\alpha)}
e(a+c-(b+c),x)=\vee_{x\in S^{-1}(\alpha)}e(-x+a+c-(b+c),0)=\vee_{x\in S^{-1}(\alpha)}e(-x+a+c,b+c)
=\vee_{x\in S^{-1}(\alpha)}e(-x+a,b)=\vee_{x\in S^{-1}(\alpha)}e(a-b,x)=\tilde{e}((a+S), (b+S))$. Therefore,
$(G+S;\tilde{e},\oplus,S)$ is an $L$-ordered group.

(ii) Clearly, $\tilde{S}$ is well defined. In fact, if $a+S=b+S$, then $S(a-b)=S(0)$, so $a-b=x$, for some
$x\in S^{-1}(\alpha)$. It follows that, $S(a)=S(x+b)\geq S(x)\wedge S(b)=\alpha \wedge S(b)=S(b)$. In a similar
way, we can show that $S(b)\geq S(a)$, whence $\tilde{S}$ is well defined. Choose $a,b\in G$.
$\tilde{S}(-(a+S))=\tilde{S}(-a+S)=S(-a)=S(a)=\tilde{S}(a+S)$.
Clearly, $\tilde{S}((a+S)\oplus(b+S))=\tilde{S}((a+b)+S)
=S(a+b)\geq S(a)\wedge S(b)=\tilde{S}(a+S)\wedge \tilde{S}(b+S)$.
If $x\in G$, then $\tilde{S}(x+S)\wedge \tilde{e}(0+S,a+S)\wedge \tilde{e}(a+S,x+S)=S(x)\wedge \vee_{t\in S^{-1}(\alpha)}e(a-x,t)\wedge \vee_{t'\in S^{-1}(\alpha)}e(t',a)$. For each $t\in S^{-1}(\alpha)$,  $S(x)=S(x+t-t)\geq S(x+t)\wedge S(t)=S(x+t)$, so
$\tilde{S}(x+S)\wedge \tilde{e}(0+S,a+S)\wedge \tilde{e}(a+S,x+S)=\vee_{t,t'\in S^{-1}(\alpha)}(S(x+t)\wedge e(a-x,t)\wedge e(t',a))=
\vee_{t,t'\in S^{-1}(\alpha)}(S(x+t)\wedge e(a,t+x)\wedge e(t',a))=\vee_{t,t'\in S^{-1}(\alpha)}(S(x+t)\wedge S(t)\wedge e(a,t+x)\wedge e(t',a))\leq S(a)$
(since $S$ is convex). It follows that $\tilde{S}(x+S)\wedge \tilde{e}(0+S,a+S)\wedge \tilde{e}(a+S,x+S)\leq S(a)=\tilde{S}(a+S)$.
By Proposition \ref{L-subgroup prop}, $\tilde{S}$ is convex.
Moreover, $\tilde{S}$ satisfies the normal condition ($\tilde{S}((a+S)\oplus(b+S)\oplus (-a+S) )=\tilde{S}((a+b-a)+S)=S(a+b-a)=S(b)=\tilde{S}(b+S)$). In addition,
for each $a\in G$, $\tilde{S}(a+S)=\alpha$  implies that $S(a)=\alpha$,
hence $a+S=0+S$ (see the note before Theorem \ref{quotient}).
Therefore, $(\tilde{S})^{-1}(\alpha)$ has exactly one element.
\end{proof}

\begin{cor}
Let $S$ be an $L$-filter of an $L$-ordered group $(G;e,+,0)$. Then the map
$\pi_{S}:G\ra G+S$, sending $g$ to $g+S$, is both a group homomorphism and monotone.
\end{cor}

By the above corollary, if $S$ is an $L$-filter of $(G;e,+,0)$, then
there exists an $L$-order relation on $(G+S;\oplus,0+S)$ such that the natural map $\pi_{S}:G\ra G+S$
is a group homomorphism and a monotone map. Moreover, $\tilde{S}$ is normal and convex.
Now, let $S$ be an $L$-subgroup of $(G;e,+,0)$ such that
there exists an $L$-order relation $f$ on $G+S$ that make $(G+S;\oplus,0+S)$ an $L$-ordered
group, $\pi_S:(G;e,+,0)\ra (G+S;f,\oplus,0+S)$ is monotone and in $(G+S;f,\oplus,0+S)$ $\tilde{S}$ is normal and convex.
We claim that $S$ is normal and convex, too. Let $x,a\in G$.
Since $\tilde{S}$ is an $L$ ideal of $(G+S;f,\oplus,0+S)$, then by Proposition \ref{Quo}(ii),
$S(a+x-a)=\tilde{S}((a+S)\oplus (x+S)\oplus (-a+S))=\tilde{S}((x+S)\oplus (a+S)\oplus (-a+S))=\tilde{S}
(x+S)=S(x)$, so $S$ is normal.
By Proposition \ref{L-subgroup prop}, it suffices to show that $S(x)\wedge e(0,a)\wedge e(a,x)\leq S(a)$.
Since $\pi_S:(G;e,+,0)\ra (G+S;f,\oplus,0+S)$ is monotone, then
$S(x)\wedge e(0,a)\wedge e(a,x)\leq S(x)\wedge f(0+S,a+S)\wedge f(a+S,x+S)=\tilde{S}(x+S)\wedge f(0+S,a+S)\wedge f(a+S,x+S)$.
By the assumption, $\tilde{S}(x+S)\wedge f(0+S,a+S)\wedge f(a+S,x+S)\leq \tilde{S}(a)=S(a)$. Therefore, $S$ is convex.

\begin{thm}\label{iso}
Let $(G;e_G,+,0)$ and $(H;e_H,+,0)$ be two $L$-ordered groups,
$f:\!\!(G;e_G,+,0) \!\!\ra\!\! (H;e_H,+,0)$ be both a group homomorphism and a monotone map.
Then the map $K(f):G\ra L$ defined by $K(f)(x)=e_{H}(0,f(x))\wedge e_{H}(f(x),0)$ is an $L$-filter of $G$ and
$\tilde{f}:G+K(f)\ra H$ is a one-to-one group homomorphism and monotone.
\end{thm}

\begin{proof}
Since $f$ is a group homomorphism, by Proposition \ref{4.3}(ii), for any $x\in G$, $K(f)(-x)=e_{H}(0,f(-x))\wedge e_{H}(f(-x),0)=e_{H}(0,-f(x))\wedge e_{H}(-f(x),0)=e_{H}(f(x),0)\wedge e_{H}(0,f(x))=K(f)(x)$. Also,
$K(f)(x+y)=e_{H}(0,f(x+y))\wedge e_{H}(f(x+y),0)=e_{H}(0,f(x)+f(y))\wedge e_{H}(f(x)+f(y),0)=e_{H}(-f(y),f(x))\wedge e_{H}(f(y),-f(x))\geq
e_{H}(-f(y),0)\wedge e_H(0,f(x))\wedge e_{H}(f(y),0)\wedge e_H(0,-f(x))=e_{H}(0,f(y))\wedge e_H(0,f(x))\wedge e_{H}(f(y),0)\wedge e_H(f(x),0)=
K(f)(x)\wedge K(f)(y)$, so $K(f)$ is an $L$-subgroup of $G$. On the other hand, for all $x,y,a\in G$, we have
$K(f)(x)\wedge K(f)(y)\wedge e_G(x,a)\wedge e_G(a,y)=e_{H}(0,f(x))\wedge e_{H}(f(x),0)\wedge e_{H}(0,f(y))\wedge e_{H}(f(y),0)
\wedge e_G(x,a)\wedge e_G(a,y)$. Since $f$ is a monotone map, then $e_G(x,a)\leq e_{H}(f(x),f(a))$ and $e_G(a,y)\leq e_H(f(a),f(y))$. It follows that
$K(f)(x)\wedge K(f)(y)\wedge e_G(x,a)\wedge e_G(a,y)\leq  e_{H}(0,f(x))\wedge e_{H}(f(x),0)\wedge e_{H}(0,f(y))\wedge e_{H}(f(y),0)\wedge e_{H}(f(x),f(a))\wedge
e_H(f(a),f(y))\leq e_{H}(0,f(a))\wedge e_{H}(f(a),0)\wedge e_{H}(f(x),0)e_{H}(0,f(y))\leq K(f)(a)$, so $K(f)$ is convex. It can be easily shown that $K(f)(a+x-a)=K(f)(x)$ for all $a,x\in G$.
Thus $K(f)$ is an $L$-filter of $(G;e_G,+,0)$. By Theorem \ref{quotient},
$(G+K(f),\tilde{e_G},\oplus,K(f))$ is an $L$-ordered group. Consider the map $\varphi :G+K(f)\ra H$ sending each element $a+K(f)\in G+K(f)$ to
$f(a)$. If $a+K(f)=b+K(f)$, then $e_H(f(a),f(b))\wedge e_H(f(b),f(a))=K(f)(a-b)=K(f)(0)=1$, whence by ${\rm (E2)}$, $f(a)=f(b)$. Clearly $\tilde{f}$ is a one-to-one
group homomorphism. Moreover, if $f(a)=f(b)$, then for each $x\in G$, $aK(f)(x)=K(f)(x-a)=e_H(0,f(x-a))\wedge e_H(f(x-a),0)=e_H(0,f(x)-f(a))\wedge e_H(f(x)-f(a),0)=e_H(0,f(x)-f(b))\wedge e_H(f(x)-f(b),0)=e_H(0,f(x-b))\wedge e_H(f(x-b),0)=bK(f)(x)$, thus $\tilde{f}$ is one-to-one.
It is easy to prove that $K(f)(x)=1$ if and only if $x\in \ker(f)=f^{-1}(0)$.
Finally, we show that $\tilde{f}$ is monotone. Take $a,b\in G$. Since $f$ is monotone, for each $t\in \ker(f)$, we have
$e_G(a-b,t)\leq e_H(f(a-b),f(t))=e_H(f(a-b),0)=e_H(f(a),f(b))$ and so $\tilde{e}(a+K(f),b+K(f))=\vee_{t\in \ker(f)}  e_G(a-b,t)\leq e_H(f(a),f(b))$.
Therefore, $\tilde{f}$ is monotone.
\end{proof}

In the next theorem, we show that the positive cone of an $L$-ordered group $(G;e,+,0)$
has particular properties and each map $S:G\ra L$ satisfying these properties is the positive cone of a
suitable $L$-order relation on $G$.

\begin{thm}\label{S-positive cone}
Let $S$ be the positive cone of $(G;e,+,0)$. Then $S$ has the following properties:
\begin{itemize}
\item[{\em (i)}] $S(x)\wedge S(y)\leq S(x+y)$ for all $x,y\in G$.
\item[{\em (ii)}] $S(x)=S(-x)=1$ implies that $x=0$ for all $x\in G$.
\item[{\em (iii)}] $S(0)=1$, and $S(x+y-x)=S(y)$ for all $x,y\in G$.
\item[{\em (iv)}] $S$ is convex.
\end{itemize}
Moreover, if $(G;+,0)$ is a group and $S:G\ra L$ is a map satisfying the properties {\rm (i)},
{\rm (ii)} and {\rm (iii)},
then there exists a map $f:G\times G\ra L$ such that $(G;f,+,0)$ is an $L$-ordered group and $S$ is convex in this $L$-ordered group.
\end{thm}

\begin{proof}
Let $S$ be the positive cone of $(G;e,+,0)$. Choose $x,y\in G$. Clearly, $S(0)=1$.
$S(x+y)=e(0,x+y)=e(-y,x)\geq e(-y,0)\wedge e(0,x)=e(0,y)\wedge e(0,x)=S(y)\wedge S(x)$.
Moreover, if $S(x)=S(-x)=1$, then $e(0,x)=1=e(0,-x)=e(x,0)$, whence by ${\rm (E3)}$, $x=0$. Now, let $a\in G$. Then
$S(x)\wedge S(y)\wedge e(x,a)\wedge e(a,y)=e(0,x)\wedge e(0,y)\wedge e(x,a)\wedge e(a,y)\leq
e(0,a)\wedge e(0,y)\wedge e(a,y)\leq S(a)$. Hence, $S$ is convex and
$S(x+y-x)=e(0,x+y-x)=e(x,x+y)=e(0,y)=S(y)$ for all $x,y\in G$. Conversely, let $S:G\ra L$ be a
map with  properties (i)--(iii). For each $a,b\in G$, define $f(a,b)=S(b-a)$.
From (iii) it follows that ${\rm (E1)}$ holds.
Let $a,b,c\in G$. If $f(a,b)=f(b,a)=1$, then $S(b-a)=S(-(b-a))=1$ and so by (ii), $b=a$.
Also, for any $x\in G$, we have $f(x+a,x+b)=S(x+b-a-x)=S(b-a)=f(a,b)$ (by (iii))  and by (i) and (iii), $f(a,b)\wedge f(b,c)=
S(b-a)\wedge S(c-b)\leq S(b-a+c-b)=S(c-a)=f(a,c)$, so $(G;f,+,0)$ is an $L$-ordered group. Moreover,
$S(x)\wedge S(y)\wedge f(x,a)\wedge f(a,y)=S(x)\wedge S(y)\wedge S(a-x)\wedge S(y-a)\leq S(x+a-x)\wedge S(y)\wedge
S(y-a)\leq S(a)$. Therefore, $S$ is convex.
\end{proof}

Let $(G;e,+,0)$ be an $L$-ordered group and $T\in L^G$ be a normal $L$-subgroup of $G$ such that $T(0)=1$. Suppose
that $T$ is not convex. Define $f:G\times G\ra L$ by $f(x,y)=T^+(y-x)$. We claim that $(G;f,+,0)$ is an
$L$-ordered group. Put $x,y,z\in G$. Then $f(x,x)=T(x-x)\wedge e(0,x-x)=T(0)\wedge 1=1$ and
$f(x,y)=f(y,x)=1$ implies that $T(y-x)\wedge e(0,y-x)=T(x-y)\wedge e(0,x-y)=1$ and so
by Proposition \ref{4.3}(i), $e(x,y)=e(y,x)=1$, whence $x=y$. Also, $f(x,y)\wedge f(y,z)=
T(y-x)\wedge e(0,y-x)\wedge T(z-y)\wedge e(0,z-y)$. Since $T$ is a normal $L$-subgroup of $(G;e,+,0)$, then
we get $f(x,y)\wedge f(y,z)\leq T(y-x+z-y)\wedge e(0,y-x) \wedge e(0,z-y)\leq T(-x+z)\wedge e(x,y)\wedge e(y,z)\leq
T(z-x)\wedge e(x,z)=T(z-x)\wedge e(0,z-x)=f(x,z)$. Moreover,
$f(x+z,y+z)=T(y+z-z-x)\wedge e(0,y+z-z-x)=f(x,z)$. So, $(G;f,+,0)$ is an $L$-ordered group.
Now, we show that $T$ is a convex $L$-subgroup of $(G;f,+,0)$. Clearly, $T$ is an $L$-subgroup.
For any $a,x\in G$, we have
$T(x)\wedge f(0,a)\wedge f(a,x)=T(x)\wedge T(a)\wedge T(x-a)\wedge T(a)$, hence by
Proposition \ref{L-subgroup prop}, $T$ is convex in $(G;f,+,0)$ (note that $T$ is also, normal in $(G;f,+,0)$).

Theorem \ref{S-positive cone} is very useful. We can use it to make many $L$-order
relations on a group (specially an Abelian group), by applying the following remark.

\begin{rmk}\label{rmk}
Let $(G;+,0)$ be a group and $S:G\ra L$ be an arbitrary map such that $S(0)=1$ and for each $x\in G$, $S(x)=1$ implies that $x=0$.

(i) If there exists $\alpha\in L-\{1\}$ such that $S(x)\leq \alpha$ for all $x\in G-\{0\}$,
then define $\overline{S}:G\ra L$, by $\overline{S}(x)=\vee_{a\in G}S(a+x-a)$. It is clear that
$\overline{S}(x)=\overline{S}(a+x-a)$ for all $a\in G$. Consider that map, $T:G\ra L$, define by
$T(x)=\vee_{\{x_1,x_2,\ldots,x_n\in G|\ x_1+x_2+\cdots+x_n=x, \ \exists\, n\in\mathbb{N} \}}
\overline{S}(x_1)\wedge \overline{S}(x_2)\wedge\cdots\wedge \overline{S}(x_n)$.
Clearly, $T(0)=1$ and for each $x\in G$, $T(x)=T(-x)=1$, implies that $x=0$.
It is easy to show that $T(x)\wedge T(y)\leq T(x+y)$, for each $x,y\in G$.
Let $x,y\in G$. Then $T(x+y-x)=\vee_{\{x_1,x_2,\ldots,x_n\in G|\ x_1+x_2+\cdots+x_n=x+y-x, \ \exists\, n\in\mathbb{N} \}}\overline{S}(x_1)\wedge \overline{S}(x_2)\wedge\cdots\wedge \overline{S}(x_n)$.
If $x_1+x_2+\cdots+x_n=x+y-x$, then $-x+x_1+x-x+x_2+x-x\cdots+x-x+x_n+x=-x+x_1+x_2+\cdots+x_n+x=y$ and so
$T(y)\geq \overline{S}(-x+x_1+x)\wedge \overline{S}(-x+x_2+x)\wedge\cdots\wedge \overline{S}(-x+x_n+x)=
\overline{S}(x_1)\wedge \overline{S}(x_2)\wedge\cdots\wedge \overline{S}(x_n)$. It follows that
$T(x+y-x)\leq T(y)$. In a similar way, we can show that $T(y)\leq T(x+y-x)$.
Hence $T$ satisfies the conditions {\rm (i)}--{\rm (iii)} of Theorem \ref{S-positive cone}.
So by applying Theorem \ref{S-positive cone}, we can make an $L$-order relation on $G$.

(ii) Let $\alpha\in L-\{1\}$ and $A$ be a subset of $G$ such that $A$ is closed under $+$,
$0\notin A$ and $x\in A$ implies that $-x\in G-A$ for all $x\in G$.
If $S(x)=\alpha$ for all $x\in A$, $S(x)\leq \alpha$,
for all $x\in G-(A\cup\{0\})$ and $T$ is the map which is defined in (i), then by (i),
$T$ satisfies on conditions {\rm (i)}--{\rm (iii)} of Theorem \ref{S-positive cone}.
Now, we define $H:G\ra L$ by
$H(x)=T(x)$ for all $x\in G-A$ and $T(x)=1$, for all $x\in A$. For each $x\in G$, $H(x)=H(-x)=1$ implies that $x=0$ and $H(0)=1$. Put $x,y\in G$.
\begin{itemize}
\item[(1)] If $x,y\in G-A$, then $H(x)\wedge H(y)=T(x)\wedge T(y)\leq T(x+y)\leq H(x+y)$.

\item[(2)] If $x\in G-A$ and $y\in A$, then $T(x)\wedge T(y)=T(x)\wedge\alpha=T(x)=H(x)\wedge\alpha=H(x)\wedge 1=H(x)\wedge H(y)$
and $T(x)\wedge T(y)\leq T(x+y)\leq H(x+y)$, so $H(x)\wedge H(y)\leq H(x+y)$.

\item[(3)] If $x\in A$ and $y\in G-A$, then similarly to (2), we can show that $H(x)\wedge H(y)\leq H(x+y)$.

\item[(4)] If $x\in A$ and $y\in A$,  by assumption $x+y\in A$, and so $H(x)\wedge H(y)\leq 1=H(x+y)$.
\end{itemize}
From (1)-(4), it follows that $H$  satisfies the conditions {\rm (i)}--{\rm (iii)} of Theorem \ref{S-positive cone}.
Applying Theorem \ref{S-positive cone}, we can make an $L$-order relation on $G$. We must note that $G$ with this $L$-order relation is
an $L$-ordered group.
\end{rmk}

In the sequel, two results about distributivity in $L$-lattice ordered groups are verified. First,
we show that if $(G;e,+,0)$ is an $L$-lattice, then one of the conditions in (\ref{dis}) implies that
$G$ is distributive. Then we find a condition under which $G$ must be distributive.

\begin{lem}
Let $(G;e,+,0)$ be an $L$-lattice ordered group. Then the following are equivalent:
\begin{itemize}
\item[{\rm (i)}] $(G;e,+,0)$ is distributive.
\item[{\rm (ii)}] $a\wedge \sqcup S=\sqcup (a\wedge S)$ for all $a\in G$ and $S\in L^{G}_{\bullet }$.
\item[{\rm (iii)}] $a\vee \sqcap S=\sqcap (a\vee S)$ for all $a\in G$ and $S\in L^{G}_{\bullet }$.
\end{itemize}
\end{lem}

\begin{proof}
(i) $\Rightarrow$ (ii) Clearly, if $(G;e,+,0)$ is distributive, then (ii) and (iii) hold.

(ii) $\Rightarrow$ (iii) Let $a\in G$ and $S$ be an arbitrary element of $L^{G}_{\bullet }$.
Then there exist $x_1,x_2,\ldots,x_n\in G$
such that $Supp(S)=\{x_1,x_2,\ldots,x_n\}$. Clearly, $Supp(-S)=\{-x_1,-x_2,\ldots,-x_n\}$, so $(-S)\in L^{G}_{\bullet }$.
Hence, we have
\begin{eqnarray*}
a\vee \sqcap S&=&a\vee\sqcap -(-S)\\&=&-(-a)\vee(-\sqcup(-S)), \mbox{ by Proposition \ref{4.3}(v)},\\
&=& -(-a\wedge (\sqcup(-S))),\mbox{ by Remark \ref{4.2}(i) and Theorem \ref{group-dis}},\\
&=& -(\sqcup(-a)\wedge (-S)), \mbox{ by (ii)},\\
&=&\sqcap-((-a)\wedge (-S)), \mbox{ by (ii)}.
\end{eqnarray*}
Now we show that $-((-a)\wedge (-S))=a\vee S$. Put $y\in G$.
\begin{eqnarray*}
-((-a)\wedge (-S))(y)&=&((-a)\wedge (-S))(-y)\\&=&\vee_{\{t\in G|\ (-a)\wedge t=-y\}}(-S)(t)\\
&=&\vee_{\{t\in G|\ (-a)\wedge -(-t)=-y\}}S(-t)\\
&=& \vee_{\{t\in G|\ -(a\vee (-t))=-y\}}S(-t), \mbox{ by Remark \ref{4.2}(i) and Theorem \ref{group-dis}},\\
&=& \vee_{\{t\in G|\ a\vee (-t)=y\}}S(-t)\\
&=& \vee_{\{u\in G|\ a\vee u=y\}}S(u)\\
&=& (a\vee S)(y).
\end{eqnarray*}
(iii) $\Rightarrow$ (i) It suffices to show that $a\wedge \sqcup S=\sqcup (a\wedge S)$,
for all $a\in G$ and $S\in L^{G}_{\bullet }$. The proof of this
part is similar to the proof of ((ii) $\Rightarrow$ (iii)).
\end{proof}

\begin{thm}\label{4.6}
Let $(G;e,+,0)$ be an $L$-ordered group, $a\in G$ and $S\in L^{G}_{\bullet }$ such that $J=\sqcup S$.
Then $a\wedge \sqcup S=\sqcup (a\wedge S)$ if and only if
$\wedge_{y\in G}[S(y)\ra e(x,(a\wedge J)-(a\wedge y))]\leq \wedge_{y\in G}(S(y)\ra e(x,J-y))$ for all $x\in G$.
\end{thm}

\begin{proof}
Suppose that $a\wedge \sqcup S=\sqcup (a\wedge S)$. By Theorem \ref{adjoint}, for all $x\in G$, we have
\begin{eqnarray*}
e(a\wedge J,x)&=&\wedge_{y\in G}[(a\wedge S(y))\ra e(y,x)]\\
&=&\wedge_{y\in G}[(\vee_{t\in G, \ a\wedge t=y}S(t))\ra e(y,x)]\\
&=& \wedge_{y\in G}\wedge_{\{t\in G|\ a\wedge t=y\}}[S(t)\ra e(y,x)]\\
&=& \wedge_{y\in G}[S(y)\ra e(a\wedge y,x)].
\end{eqnarray*}
It follows from Proposition \ref{4.3}(i) that
$$e(x,0)=e(a\wedge J,-x+(a\wedge J))=\wedge_{y\in G}[S(y)\ra e(a\wedge y,-x+(a\wedge J))]=\wedge_{y\in G}[S(y)\ra e(x,(a\wedge J)-(a\wedge y))].$$
Also, by Theorem \ref{join and meet}(i), $e(J,x)=\wedge_{y\in G}(S(y)\ra e(y,x))$ for all $x\in G$, so
\begin{eqnarray}
\label{R4.6} e(x,0)=e(0,-x)=e(J,-x+J)=\wedge_{y\in G}(S(y)\ra e(y,-x+J))=\wedge_{y\in G}(S(y)\ra e(x,J-y)).
\end{eqnarray}
Therefore, $\wedge_{y\in G}(S(y)\ra e(x,J-y))=\wedge_{y\in G}[S(y)\ra e(x,(a\wedge J)-(a\wedge y))]$ for all $x\in G$. Conversely, let
$\wedge_{y\in G}[S(y)\ra e(x,(a\wedge J)-(a\wedge y))]\leq \wedge_{y\in G}(S(y)\ra e(x,J-y))$ for all $x\in G$.
By Proposition \ref{4.3}, it suffices to show that
$\sqcap T=0$, where $T=(a\wedge J)+(-(a\wedge S))$ (By Proposition \ref{4.3}(iv) and (v),  $\sqcap T$ exists).
For each $y\in G$, we have
\begin{eqnarray*}
T(y)=\vee_{\{t\in G| \ (a\wedge J)+t=y\}}(-(a\wedge S)(t))&=&(-(a\wedge S))(-(a\wedge J)+y)\\
&=&(a\wedge S)(-y+(a\wedge J))\\
&=&\vee_{\{t\in G| \ a\wedge t=-y+(a\wedge J)\}}S(t)\\
&=&\label{R2.4.5}\vee_{\{t\in G| \ y=(a\wedge J)-( a\wedge t)\}}S(t).
\end{eqnarray*}
It follows that
\begin{eqnarray*}
\wedge_{y\in G}T(y)\ra e(x,y)&=&\wedge_{y\in G}[(\vee_{t\in G, \ y=(a\wedge J)-( a\wedge t)}S(t))\ra e(x,y)]\\
&=& \wedge_{y\in G}\wedge_{t\in G, \ y=(a\wedge J)-( a\wedge t)} (S(t)\ra e(x,y))\\
&=&\wedge_{y\in G} [S(y)\ra e(x,(a\wedge J)-( a\wedge y))]\\
&\leq& \wedge_{y\in G}(S(y)\ra e(x,J-y)), \mbox{ by assumption, }\\
&=&e(x,0),\mbox{ by (\ref{R4.6})}.
\end{eqnarray*}
Now we show that $T(x)\leq e(0,x)$ for all $x\in G$. For all $x\in G$, $T(x)=\vee_{t\in G, \ x=(a\wedge J)-( a\wedge t)}S(t)$. It suffices to prove that
if $x=(a\wedge J)-( a\wedge t)$, then $S(t)\leq e(0,x)$ for all $t\in G$, or equivalently,
$S(t)\leq e(0,(a\wedge J)-( a\wedge t))=e(a\wedge t,a\wedge J)$ for all $t\in G$.
Since $\sqcup S=J$, then $1=e(J,J)=\wedge_{t\in G}(S(t)\ra e(t,J))$, so
$S(t)\leq e(t,J)$ for all $t\in G$.
Hence, by Proposition \ref{3.4.2}(iii), we have $S(t)\leq e(t,J)\leq e(a\wedge t,a\wedge J)$. Therefore, $\sqcap T=0$.
\end{proof}

There is a well known result on lattice ordered group, which prove that if $x,y$ are elements of lattice
ordered group $(G;.,1)$ such that $x^n\leq y^n$, for some $n\in\mathbb{N}$, then $x\leq y$ (see \cite[Theorem 9.11]{Blyth}). In the next theorem
we try to extend this result for $L$-lattice ordered groups.

\begin{thm}\label{4.7}
Let $(G;e,+,0)$ be an $L$-lattice ordered group.
\begin{itemize}
\item[{\rm(i)}] For each $z\in G$ and $n\in \mathbb{N}$, $e(z,0)=e(nz\vee 0,(n-1)(z\vee 0))$.
\item[{\rm(ii)}] If $x,y\in G$ such that $x+y=y+x$ and $n\in \mathbb{N}$, then $e(x,y)=e(nx\vee ny, (n-1)(x\vee y)+y)$.
\end{itemize}
\end{thm}
\begin{proof}
(i) Put $z\in G$. By Remark \ref{4.2}, $(G;\leq_e)$ is a lattice ordered group. Since $0+z=z+0$, it is easy
to see that $m(z\vee 0)=mz\vee (m-1)z\vee (m-2)z\vee \cdots \vee z\vee 0$ for all $m\in \mathbb{N}$.
\begin{eqnarray*}
e(z,0)&=&e(z,0)\wedge e(0,0)\\&=&e(0\vee z,0), \mbox{ by Proposition \ref{3.4.2}(ii), }\\
&=& e(n(z\vee 0),(n-1)(z\vee 0)), \mbox{ by Proposition \ref{4.3}(i)}\\
&=& e(nz\vee (n-1)z\vee\cdots \vee z\vee 0,(n-1)(z\vee 0))\\
&=& e(nz\vee (n-1)z\vee\cdots \vee z\vee 0\vee 0,(n-1)(z\vee 0))\\
&=& e((nz\vee 0)\vee (n-1)(z\vee 0),(n-1)(z\vee 0))\\
&=& e(nz\vee 0,(n-1)(z\vee 0))\wedge e((n-1)(z\vee 0),(n-1)(z\vee 0)), \mbox{ by Proposition \ref{3.4.2}(ii), }\\
&=& e(nz\vee 0,(n-1)(z\vee 0)).
\end{eqnarray*}
(ii) Let $x,y\in G$. Then by (i), we  have
\begin{eqnarray*}
e(x,y)&=&e(x-y,0)=e(n(x-y)\vee 0,(n-1)((x-y)\vee 0))= e(nx-ny\vee 0,(n-1)((x\vee y)-y))\\
&=& e((nx\vee ny)-ny,(n-1)(x\vee y)-(n-1)y)= e(nx\vee ny, (n-1)(x\vee y)+y).
\end{eqnarray*}
Note that, the stated result before Theorem \ref{4.7} can be easily obtained from Theorem \ref{4.7}.
In fact, if $x^n\leq y^n$, for some $n\in\mathbb{N}$, then by Theorem \ref{4.7}(i),
$e(x,y)=e(nx\vee ny, (n-1)(x\vee y)+y)=e(ny,(n-1)(x\vee y)-y)=e((n-1)y,(n-1)(x\vee y))=1$ and so $x\leq y$
\end{proof}

\begin{thm}\label{4.8}{\rm (Riesz's theorem for $L$-lattice ordered groups)}
Let $(G;e,+,0)$ be an $L$-lattice ordered group and $t\in L$. If $a,b_1,b_2,\ldots, b_n\in G$ such that
$t\leq e(0,a)\wedge e(0,b_i)\wedge e(a,b_1+b_2+\cdots +b_n)$ for all $i\in\{1,2,\ldots, n\}$,
then there exist $a_1,a_2,\ldots,a_n\in G$ such that $a=a_1+a_2+\cdots+a_n$ and
$t\leq e(0,a_i)\wedge e(a_i,b_i)$ for all $i\in\{1,2,\ldots, n\}$.
\end{thm}

\begin{proof}
The result is established by induction. If $n=1$, the result is clear (in fact, $a_1=a$).
Let $1\leq n$ and the result is true for $n$.
Suppose that $b_1,b_2,\ldots, b_{n+1}\in G$ such that
$t\leq e(0,a)\wedge e(0,b_i)\wedge e(a,b_1+b_2+\cdots+ b_{n+1})$ for all $i\in\{1,2,\ldots, n+1\}$.
From (E2) and $t\leq e(0,b_1)\wedge e(0,b_2)=e(0,b_1)\wedge e(b_1,b_1+b_2)$ it follows that $t\leq e(0,b_1+b_2)$. In a similar way, we can show that
\begin{eqnarray}
\label{Riesz 1} t\leq e(0,b_1+b_2+\cdots+b_n)
\end{eqnarray}
So by assumption, for all $i\in\{1,2,\ldots,n\}$  we have
\begin{eqnarray*}
t&\leq& e(0,a)\wedge e(0,b_i)\wedge e(a,b_1+b_2+\cdots+ b_{n+1})\\
&=& e(0,a)\wedge e(0,b_i)\wedge e(a-b_{n+1},b_1+b_2+\cdots+ b_{n}), \mbox{ by Proposition \ref{4.3}(i)}\\
&=& e(0,a)\wedge e(0,b_i)\wedge e(a-b_{n+1},b_1+b_2+\cdots+ b_{n})\wedge e(0,b_1+b_2+\cdots+b_n), \mbox{ by (\ref{Riesz 1})}\\
&=& e(0,a)\wedge e(0,b_i)\wedge e(0\vee a-b_{n+1},b_1+b_2+\cdots+ b_{n}), \mbox{ by Proposition \ref{3.4.2}(ii)}.
\end{eqnarray*}
Since $(G;e,+,0)$ is an $L$-lattice ordered group, by Remark \ref{4.2}, $(G;\leq_e)$ is a crisp one and so
$0\vee a-b_{n+1}=a-(a\wedge b_{n+1})$. Set $u:=0\vee a-b_{n+1}$. Then
$t\leq e(0,u)\wedge e(0,b_i)\wedge  e(u,b_1+b_2+\cdots+ b_{n})$, for all $i\in\{1,2,\ldots,n\}$ (since $e(0,u)=1$), hence
by the induction hypothesis there are $a_1,a_2,\ldots,a_n\in G$ such that $u=a_1+a_2+\cdots+a_n$ and
$t\leq e(0,a_i)\wedge e(a_i,b_i)$ for all $i\in\{1,2,\ldots, n\}$. Thus $a=a_1+a_2+\cdots+a_n+(a\wedge b_{n+1})$, set $a_{n+1}=a\wedge b_{n+1}$.
Then $e(a_{n+1},b_{n+1})=1$ and by Proposition \ref{3.4.2}(i),
$e(0,a_{n+1})=e(0,a\wedge b_{n+1})=e(0,a)\wedge e(0,b_{n+1})\geq t$. Therefore,
the result is obtained.
\end{proof}

\begin{cor}\label{4.9}
Let $(G;e,+,0)$ be an $L$-lattice ordered group and $a,b,c\in G$. Then for each
$t\in L$, with $t\leq e(0,a)\wedge e(0,b)\wedge e(0,c)\wedge e(0,a\wedge (b+c))$, we have
$t\leq e(a\wedge (b+c),(a\wedge b)+(a\wedge c))$.
\end{cor}

\begin{proof}
Let $t\leq e(0,a)\wedge e(0,b)\wedge e(0,c)$. By
Proposition \ref{3.4.2}(i), $e(0,a\wedge (b+c))=e(0,a)\wedge e(0,b+c)\geq e(0,a)\wedge e(0,b)\wedge
e(b,b+c)=e(0,a)\wedge e(0,b)\wedge e(0,c)\geq t$.
Also,
$e(a\wedge (b+c),b+c)=1$, so by Theorem \ref{4.8}, there exist $x_1,x_2\in G$ such that
$a\wedge (b+c)=x_1+x_2$, $t\leq e(0,x_1)\wedge e(x_1,b)$ and $t\leq e(0,x_2)\wedge e(x_2,b)$.
Hence, $e(0,x_1)\wedge e(x_1,a)=e(0,x_1)\wedge e(x_1,x_1+x_2)= e(0,x_1)\wedge e(0,x_2)\geq t$, whence
$t\leq e(x_1,a)$ and  by Proposition \ref{3.4.2}(i), $t\leq e(x_1,a)\wedge e(x_1,b)=e(x_1,a\wedge b)$.
In a similar way, we can show that $t\leq e(x_2,a\wedge c)$. Therefore,
\begin{eqnarray*}
e(a\wedge (b+c),(a\wedge b)+(a\wedge c))&=&e(x_1+x_2,(a\wedge b)+(a\wedge c))\\
&\geq& e(x_1+x_2,(a\wedge b)+x_2)\wedge e((a\wedge b)+x_2,(a\wedge b)+(a\wedge c))\\
&=&e(x_1,a\wedge b)\wedge e(x_2,a\wedge c)\geq t, \mbox{ by Proposition \ref{4.3}(i)}.
\end{eqnarray*}
\end{proof}






\begin{thebibliography}{99}
\footnotesize{

\bibitem{anderson-feil} M. Anderson, T. Feil, Lattice-Ordered Groups: An Introduction, {\it D. Reidel Publishing Co.},
Dordrecht, 1988.

\bibitem{G1} R. B\v{e}lohl\'{a}vek, U. Bodenhofer, P. Cintula, Relations in fuzzy class theory: Initial steps,
{\it Fuzzy Sets Syst.} {\bf 159} (2008), 1729--1772.

\bibitem{Bel1} R. B\v{e}lohl\'{a}vek, Fuzzy Relational Systems: Foundations and Principles, {\it Kluwer Acad. Publ.}, New York, 2002.

\bibitem{Bel2} R. B\v{e}lohl\'{a}vek, Concept lattices and order in fuzzy logic, {\it Ann. Pure Appl. Logic.} {\bf 128} (2004), 277--298.

\bibitem{Blyth} T. S. Blyth, Lattices and Ordered Algebraic Structures,  {\it Springer-Verlag}, London, 2005.

\bibitem{F1} C. Coppola, G. Gerla, T. Pacelli, Convergence and fixed points by fuzzy orders,
{\it Fuzzy Sets Syst.} {\bf 159} (2008), 1178--1190.

\bibitem{G2} M. Demirci, A theory of vague lattices based on many-valued equivalence relations I: General representation results,
{\it Fuzzy Sets Syst.} {\bf 151} (2005), 437--472.

\bibitem{G3} M. Demirci, A theory of vague lattices based on many-valued equivalence relations II: Complete lattices,
 {\it Fuzzy Sets Syst.} {\bf 151} (2005), 473--489.

\bibitem{Fan} L. Fan, A new approach to quantitative domain theory, {\it Electronic Notes  Theor. Comp. Sci.}
{\bf 45} (2001), 77--87.

\bibitem{Fuentes} R. Fuentes-Gonz\'{a}lez, Down and up operators associated to fuzzy relations and t-norms:
a definition of fuzzy semi-ideals, {\it Fuzzy Sets Syst.} {\bf 117} (2001), 377--389.

\bibitem{Hol} U. H\"{o}hle, N. Blanchard, Partial ordering in $L$-underdeterminate sets, {\it Inform. Sci.} {\bf 35} (1985), 133--144.

\bibitem{Johnston} P.T. Johnstone, Stone Spaces, {\it Cambridge University Press}, Cambridge, 1982.

\bibitem{F2} G.J. Klir, B. Yuan, Fuzzy Sets and Fuzzy Logic: Theory and Applications, {\it Prentice-Hall}, Upper Saddle River, 1995.

\bibitem{Mer} J. N. Mordeson, D. S. Malik, fuzzy commutative algebras, {\it World Scientific Publishing Co. Pte. Ltd.} (1998).



\bibitem{F3} S.V. Ovchinnikov, Structure of fuzzy binary relations, {\it Fuzzy Sets Syst.} {\bf 6} (1981) 169--195.

\bibitem{F4} P. Venugopalan, Fuzzy ordered sets, {\it Fuzzy Sets Syst.} {\bf 46} (1992), 221--226.

\bibitem{XZF} W. Xiea, Q. Zhang, L. Fand, The Dedekind-MacNeille completions for fuzzy posets,
{\it Fuzzy Sets Syst.} {\bf 160} (2009), 2292--2316.

\bibitem{Y1} W. Yao, Quantitative domains via fuzzy sets: Part I: Continuity of fuzzy directed complete posets,
{\it Fuzzy Sets Syst.} {\bf 161} (2010), 973--987.

\bibitem{Yao-Shi} W. Yao, F. G. Shi, Quantitative domains via fuzzy sets: Part II: Fuzzy Scott topology
on fuzzy directed-complete posets, {\it Fuzzy Sets Syst.} {\bf 173} (2011), 60--80.

\bibitem{Y2} W. Yao, An approach to fuzzy frames via fuzzy posets, {\it Fuzzy Sets Syst.} {\bf 166} (2011), 75--89.

\bibitem{YL} W. Yao, L. X. Lu, Fuzzy Galois connections on fuzzy posets, {\it Math. Log. Quart.} {\bf 55} (2009), 105--112.

\bibitem{YSI} W. Yao, Quantitative domains via fuzzy sets: Part I: continuity of fuzzy directed complete posets,
{\it Fuzzy Sets Syst.} {\bf 161} (2010), 973--987.


\bibitem{Y2S} W. Yao, F. G. Shi, Quantitative domains via fuzzy sets: Part II: Fuzzy Scott topology
on fuzzy directed-complete posets, {\it Fuzzy Sets Syst.} {\bf 173} (2011), 60--80.

\bibitem{Zadeh} L.A. Zadeh, Similarity relations and fuzzy orderings, {\it Inf. Sci.} {\bf 3} (1971), 177--200.

\bibitem{ZF} Q. Y. Zhang, L. Fan, Continuity in quantitative domains, {\it Fuzzy Sets Syst.} {\bf 154} (2005), 118--131.

\bibitem{ZL}D. X. Zhang, Y. M. Liu, $L$-fuzzy version of Stones representation theorem for distributive lattices,
{\it Fuzzy Sets Syst.} {\bf 76} (1995), 259--270.

\bibitem{ZF2} Q. Y. Zhang, W. X. Xie, Section-retraction-pairs between fuzzy domains, {\it Fuzzy Sets Syst.} {\bf 158} (2007), 99--114.

\bibitem{ZXF} Q. Y. Zhang, W. X. Xie, L. Fan, Fuzzy complete lattices, {\it Fuzzy Sets Syst.} {\bf 160} (2009), 2275--2291.



}
\end{thebibliography}
\end{document}